\documentclass[a4paper,11pt]{article}

\usepackage{verbatim}

\usepackage[T1]{fontenc}
\usepackage[latin1]{inputenc}
\usepackage{setspace}
\usepackage{indentfirst}

\usepackage{geometry}
\usepackage[hidelinks]{hyperref}

\usepackage{amsthm}
\usepackage{amsmath}
\usepackage{amsfonts}
\usepackage{amssymb}
\usepackage{mathdots}

\usepackage{times}
\usepackage{bm}
\usepackage{dutchcal}

\usepackage{tikz-cd}
\usepackage{tikz}
\usetikzlibrary{matrix,arrows,decorations.pathmorphing}
\usetikzlibrary{shapes.geometric}
\usepackage[usestackEOL]{stackengine}
\usepackage{fancybox}
\usepackage{caption}

\usepackage[toc,page]{appendix}
\usepackage{todonotes}

\DeclareMathOperator{\Hom}{Hom}
\DeclareMathOperator{\diag}{diag}
\DeclareMathOperator{\im}{Im}
\DeclareMathOperator{\id}{Id}
\DeclareMathOperator{\vol}{Vol}
\DeclareMathOperator{\pr}{pr}
\DeclareMathOperator{\SL}{SL}
\DeclareMathOperator{\PSL}{PSL}
\DeclareMathOperator{\SO}{SO}
\DeclareMathOperator{\Spin}{Spin}
\DeclareMathOperator{\Sp}{Sp}
\DeclareMathOperator{\wt}{wt}
\DeclareMathOperator{\hw}{hw}
\DeclareMathOperator{\lcm}{lcm}
\DeclareMathOperator{\Tr}{Tr}
\DeclareMathOperator{\Lie}{Lie}

\usepackage{theoremref}
\newtheorem{Thm}{Theorem}[section]

\newtheorem{Pro}[Thm]{Proposition}
\newtheorem{Lem}[Thm]{Lemma}
\newtheorem{Cor}[Thm]{Corollary}

\theoremstyle{definition}
\newtheorem{Def}[Thm]{Definition}
\newtheorem{Ex}[Thm]{Example}
\newtheorem{Rmk}[Thm]{Remark}

\newcommand\nnfootnote[1]{%
  \begin{NoHyper}
  \renewcommand\thefootnote{}\footnote{#1}%
  \addtocounter{footnote}{-1}%
  \end{NoHyper}
}

\usepackage{tikzsymbols}

\begin{document}
\title{Langlands Duality and Poisson-Lie Duality\\ via Cluster Theory and Tropicalization}
\author{Anton Alekseev\and Arkady Berenstein\and Benjamin Hoffman\and Yanpeng Li}

\newcommand{\Addresses}{{
  \bigskip
  \footnotesize

  \textsc{Section of Mathematics, University of Geneva, 2-4 rue du Li\`evre, c.p. 64, 1211 Gen\`eve 4, Switzerland}\par\nopagebreak
  \textit{E-mail address}: \texttt{Anton.Alekseev@unige.ch}

  \medskip

  \textsc{Department of Mathematics, University of Oregon, Eugene, OR 97403, USA}\par\nopagebreak
  \textit{E-mail address}: \texttt{arkadiy@uoregon.edu}

  \medskip

  \textsc{Department of Mathematics, Cornell University, 310 Malott Hall, Ithaca, NY 14853, USA}\par\nopagebreak
  \textit{E-mail address}: \texttt{bsh68@cornell.edu}
  
  \medskip

  \textsc{Section of Mathematics, University of Geneva, 2-4 rue du Li\`evre, c.p. 64, 1211 Gen\`eve 4, Switzerland}\par\nopagebreak
  \textit{E-mail address}: \texttt{yanpeng.li@unige.ch}

}}
\date{}
\maketitle

\nnfootnote{\emph{Keywords:} Langlands dual, Poisson-Lie dual, Cluster Algebras, Potentials, Tropicalization}

\begin{abstract}
    Let $G$ be a connected semisimple Lie group. There are two natural duality constructions that assign to it the Langlands dual group $G^\vee$ and the Poisson-Lie dual group $G^*$. The main result of this paper is the following relation between these two objects: the integral cone defined by the cluster structure and the Berenstein-Kazhdan potential on the double Bruhat cell $G^{\vee; w_0, e} \subset G^\vee$ is isomorphic to the integral Bohr-Sommerfeld cone defined by the Poisson structure on the partial tropicalization of $K^* \subset G^*$ (the Poisson-Lie dual of the compact form $K \subset G$). By \cite{BKII}, the first cone parametrizes the canonical bases of irreducible $G$-modules. The corresponding points in the second cone belong to integral symplectic leaves of the partial tropicalization labeled by the highest weight of the representation.
    As a by-product of our construction, we show that symplectic volumes of generic symplectic leaves in the partial tropicalization of $K^*$ are equal to symplectic volumes of the corresponding coadjoint orbits in $\Lie(K)^*$.

    To achieve these goals, we make use of (Langlands dual) double cluster varieties defined by Fock and Goncharov \cite{FG2}. These are pairs of cluster varieties whose seed matrices are transpose to each other. There is a naturally defined isomorphism between their tropicalizations. The isomorphism between the cones described above is a particular instance of such an isomorphism associated to the double Bruhat cells $G^{w_0, e} \subset G$ and $G^{\vee; w_0, e} \subset G^\vee$.
\end{abstract}

\tableofcontents

\section{Introduction}

Let $K$ be a compact connected semisimple Lie group. There are two very interesting duality constructions which involve $K$. First, one can associate to it the Langlands dual group $K^\vee$ corresponding to the root system dual to the one of $K$. Second, the group $K$ carries the standard Poisson-Lie structure $\pi_{K}$. As a Poisson-Lie group, it admits the dual Poisson-Lie group $K^*$.

The group $K^\vee$ is a compact connected semisimple Lie group while the group $K^*$ is solvable. Despite this fact, they share some common features. Let $T \subset K$ be a maximal torus of $K$ and $\mathfrak{t}=\Lie(T)$ be its Lie algebra. The Lie algebra $\mathfrak{t}^\vee=\Lie(T^\vee)$ of the maximal torus $T^\vee \subset K^\vee$ is in a natural duality with $\mathfrak{t}$:
\[
    \mathfrak{t}^\vee \cong \Hom(S^1,T)^*\otimes_{\mathbb{Z}} \mathbb{R}\cong 
    \Hom(T,S^1)\otimes_{\mathbb{Z}} \mathbb{R} \cong \mathfrak{t}^*.   \]
The Lie group $K^*\cong N_- A$ is isomorphic to a semi-direct product of the maximal nilpotent subgroup $N_- \subset G=K^\mathbb{C}$ and the abelian group $A=\exp(\sqrt{-1} \mathfrak{t})$. The Lie algebra 
\[
    \sqrt{-1} \mathfrak{t} \cong \mathfrak{t}^*
\]
plays a role analogous to the one of Cartan subalgebra for the group $K^*$. The isomorphism above is induced by the invariant scalar product on $\mathfrak{k}=\Lie(K)$ used to define the standard Poisson structures on $K$ and $K^*$. 

Furthermore, both the Langlands dual group and the Poisson-Lie dual group can be used to parametrize representations of $K$ (or finite dimensional representations of $G=K^\mathbb{C}$). On one hand, by the Borel-Weil-Bott Theorem, geometric quantization of coadjoint orbits passing through dominant integral weights in $\mathfrak{t}^*$   yields all  irreducible representations of $K$. By the Ginzburg-Weinstein Theorem \cite{GW}, the Poisson spaces $K^*$ and $\mathfrak{k}^*$ are isomorphic to each other and we can extend the Borel-Weil-Bott result to $K^*$, where for a dominant integral weight $\lambda \in \mathfrak{t}^* \cong \sqrt{-1} \mathfrak{t}$ we consider the $K$-dressing orbit in $K^*$ passing through $\exp(\lambda)$.

On the other hand, beginning with the Borel subgroup $B_-^\vee \subset G^\vee =(K^\vee)^\mathbb{C}$, Berenstein and Kazhdan \cite{BKII} constructed an integral polyhedral cone $\mathcal{C}_{BK}^{G^\vee}$ together with a tropical highest weight map 
\[
    {\hw^\vee}^t: \mathcal{C}_{BK}^{G^\vee} \to  \mathfrak{t}^\vee\cong \mathfrak{t}^*.
\]
Fibers of ${\hw^\vee}^t$ parametrize canonical bases of irreducible finite dimensional representations of $G$.  

It is the goal of this paper to establish a relation between the two duality constructions described above. 

There are several tools that we are using to this effect. First, following \cite{FG2} we introduce the notion of a double cluster variety, which is a pair $\mathcal{A}$ and $\mathcal{A}^\vee$ of cluster varieties whose seed matrices are transpose to one another. There exists  natural maps between the corresponding cluster charts. After tropicalization, all of these maps match and give rise to a global duality map between the tropical cluster varieties 
\[
    \psi: \mathcal{A}^t \to (\mathcal{A}^\vee)^t. 
\]

Our main example is that of the double Bruhat cells $G^{w_0, e} \subset B_-$ and
$G^{\vee; w_0, e} \subset B^\vee_-$. In this case, the relationship between tropicalizations can be further improved: each comes equipped with a potential function called $\Phi_{BK}$ and $\Phi^\vee_{BK}$, respectively. These cut out polyhedral cones $\mathcal{C}_{BK}^G$ and $\mathcal{C}_{BK}^{G^\vee}$ in the tropical varieties. We show that our comparison map $\psi$ maps one of these cones into the other, and preserves their Kashiwara crystal structure (up to some scaling). This gives a new perspective on a result of Kashiwara \cite{MK96} and Frenkel-Hernandez \cite{FH}.

For the discussion of the Poisson-Lie dual $K^*$, we turn to the notion of partial tropicalization that we introduced in \cite{ABHL}. The partial tropicalization $PT(K^*)$ of $K^*$ is a product $\mathcal{C}^{G}_{BK}(\mathbb{R}) \times T$ of the real Berenstein-Kazhdan cone $\mathcal{C}^{G}_{BK}(\mathbb{R})$ and a torus $T$. It comes equipped with a constant Poisson structure which induces integral affine structures on symplectic leaves. Together with the structure of the weight lattice of $K$, they define a natural Bohr-Sommerfeld lattice $\Lambda\subset \mathcal{C}^G_{BK}(\mathbb{R})$.

We show that 
\[
    \psi(\Lambda) = \mathcal{C}_{BK}^{G^\vee}.
\]
That is, the integral Bohr-Sommerfeld cone $\Lambda$ defined by the Poisson-Lie data on $K^*$ is isomorphic to the integral cone $\mathcal{C}_{BK}^{G^\vee}$ defined by the cluster structure and the potential $\Phi_{BK}^\vee$ on the double Bruhat cell $G^{\vee; w_0, e}\subset G^\vee$. The isomorphism is given by the tropical duality map of the double cluster variety.

In more detail, the cone $\mathcal{C}_{BK}^{G^\vee}$ parametrizes canonical bases of irreducible $G$-modules. For the representation with highest weight $\lambda$, the canonical basis in $V_\lambda$ is parametrized by the points of $\hw^{-t}(\lambda) \subset \mathcal{C}_{BK}^{G^\vee}$, where $\hw^t$ is the tropical highest weight map. The preimage of this set under the duality map $\psi$ is exactly the set of points of $\Lambda$ which belongs to the integral symplectic leaf in the partial tropicalization of $K^*$ corresponding to the weight $\lambda$. The relations are depicted in Figure \ref{fig:relations}.

\begin{figure}[htb]
    \[
    \begin{tikzcd}[column sep=3em, row sep=2.5em]
        \framebox{\Longstack[c]{Integral\\ coadjoint orbits $\mathcal{O}(\lambda)$}} \arrow[r, leftrightarrow]  &\framebox{\Longstack[c]{Irreducible $G$-modules}} \arrow[r, leftrightarrow] \arrow[d, leftrightarrow] & \framebox{\Longstack[c]{Fibers of\\ $\hw^{t}\colon \mathcal{C}^{G^\vee}_{BK}\to \mathfrak{t}^*$}}  \\
         & \framebox{\Longstack[c]{Integral symplectic\\ leaves of $PT(K^*)$}} & 
    \end{tikzcd}
    \]
	\caption[Caption for the list of figures]{}
	\label{fig:relations}
\end{figure}

As a by-product of our construction, we show that for generic $\lambda \in \mathfrak{t}^*$ the volume of the symplectic leaf in $PT(K^*)$ coincides with the volume of the coadjoint orbit in $\mathfrak{k}^*$ passing through $\lambda$. 

The results of this article have already been used in applications. In \cite{AHLL}, Jeremy Lane and three of the authors use the partial tropicalization $PT(K^*)$ to study the $s\to -\infty$ limiting behavior a family of symplectic forms $\omega^s_\lambda$ on the regular coadjoint orbit $K\cdot \lambda\cong K/T$. In a forthcoming work \cite{AHLL2}, the same authors use $PT(K^*)$ to establish tight lower bounds on the Gromov width of regular coadjoint orbits of $K$, as well as all other multiplicity-free $K$-spaces with regular moment map image.

The paper is organized as follows. In Sections \ref{backgroundsection} and \ref{positivesection} we recall the relevant background material on algebraic groups and positivity theory. In Section \ref{Doubleseed} we define the notion of a double cluster variety and introduce the comparison map $\psi$. Our main application is in Section \ref{potentialsection}, where we consider double Bruhat cells as double cluster varieties. Here we also consider the image of the BK cone and its crystal structure under the comparison map. Finally, in Section \ref{poissonsection} we recall the partial tropicalization $PT(K^*)$ of a dual Poisson-Lie group, and consider the Poisson geometry of $PT(K^*)$ in light of our previous discussion. In particular, we derive the relationship between $PT(K^*)$ and the representation theory of $G$ described above. In Appendix \ref{appendix1} we describe explicitly the comparison map between the cones for ${\rm SO}_{2n+1}$ and its Langlands dual group ${\rm Sp}_{2n}$. In Appendix \ref{appendix2} we describe the Bohr-Sommerfeld lattice in the more general context of tropical Poisson varieties.

{\bf Acknowledgements.} We are grateful to B. Elek, V. V. Fock, A. Goncharov, J. Lane, J. H. Lu and M. Semenov-Tian-Shansky for their useful comments and discussions, and to D. R. Youmans for his helpful comments on an earlier draft. Research of A.A. and Y.L. was supported in part by the grant MODFLAT of the European Research Council (ERC), by the grants number 178794 and 178828 of the Swiss National Science Foundation (SNSF) and by the NCCR SwissMAP of the SNSF. B.H. was supported by the National Science Foundation Graduate Research Fellowship under Grant Number DGE-1650441. A.B. and B.H. express their gratitude for hospitality and support during their visits to Switzerland in 2017 and 2018.

\section{Background on Semisimple Algebraic Groups}
\label{backgroundsection}

Let $A=[a_{ij}]$ be a symmetrizable Cartan matrix, where $i,j\in\bm{I}=\{1, \dots, r\}$. That is, $a_{ii}=2$ and $a_{ij}\in \mathbb{Z}_{\leqslant 0}$ for $i\neq j$, and there exists a sequence of positive integers ${\bm d}=\{ d_1, \dots, d_r\}$ called a symmetrizer so that $a_{ij}d_j=a_{ji}d_i$. The matrix $AD$ is positive-definite, where $D=\diag(d_1, \dots, d_r)$, and $(AD)^T=AD$. Clearly, $A$ is symmetrizable in the usual sense via:
\[\frac{d}{d_i}a_{ij}=\frac{d}{d_j}a_{ji}\]
for any natural number $d$ divisible by all $d_i$.

Let $\mathfrak{g} = \mathfrak{g}(A)$ be the semisimple Lie algebra over $\mathbb{Q}$ corresponding to the Cartan matrix $A$. Recall that $\mathfrak{g}$ is generated by $\{E_i,F_i\}_{i=1}^r$ subject to the Serre relations \cite{Kac}. Denote by $\alpha_i^\vee=[E_i, F_i]$ the $i^{\text{~th}}$ simple coroot and by $\mathfrak{h}$ the span of all simple coroots. Let $\mathfrak{h}^*$ be the linear dual space and choose a basis of simple roots $\alpha_1,\dots,\alpha_r\in \mathfrak{h}^*$ such that
\begin{equation}
\label{rootcorootpairing}
    \langle \alpha_j, \alpha_i^\vee\rangle =a_{ij}.
\end{equation}
Using this definition and a chosen symmetrizer ${\bm d}$, we can define a symmetric bilinear form on $\mathfrak{h}$:
\[
    (\alpha_i^\vee, \alpha_j^\vee):=a_{ij} d_j.
\]
This form uniquely extends to a $\mathfrak{g}$-invariant symmetric bilinear form on $\mathfrak{g}$, and induces a symmetric bilinear form on $\mathfrak{h}^*$:
\[
    (\alpha_i, \alpha_j) = d_i^{-1} a_{ij},
\]
as well as an isomorphism $\psi\colon \mathfrak{h} \to \mathfrak{h}^*$ such that
\[
    \psi(\alpha_i^\vee) = d_i \alpha_i.
\]
The formulas above imply the following standard identities:
\[
    d_i=\frac{a_{ii}}{(\alpha_i, \alpha_i)} = \frac{2}{(\alpha_i, \alpha_i)},\quad a_{ij}=d_i (\alpha_i, \alpha_j) =2\frac{(\alpha_i, \alpha_j)}{(\alpha_i, \alpha_i)}, \quad \psi(\alpha_i^\vee)=\frac{2 \alpha_i}{(\alpha_i, \alpha_i)}.
\]
Fix a positive integer $d$ such that each $d_i$ divides $d$ (for instance, we can choose $d=\lcm\{d_1,\dots,d_i\}$). Note that ${\bm d}^\vee:=\{d_i^\vee:=d/d_i\}$ defines a symmetrizer for the transposed Cartan matrix $A^\vee=A^{T}=[a_{ji}]$. Indeed,
\[
    A^\vee D^\vee=dA^{T}D^{-1}=d(D^{-1} A D)D^{-1}=dD^{-1}A=(A^\vee D^\vee)^{T}.
\]
Define the dual Lie algebra $\mathfrak{g}^\vee=\mathfrak{g}(A^\vee)$ with generators $E_i^\vee, F_i^\vee$, and choose the standard identification $\mathfrak{h}^\vee = \mathfrak{h}^*$ via
\[
    [E_i^\vee, F_i^\vee] = \alpha_i.
\]
The symmetrizer ${\bm d}^\vee$ defines new symmetric bilinear forms $(\cdot, \cdot)^\vee$ on $\mathfrak{h} = (\mathfrak{h}^\vee)^*$ and $\mathfrak{h}^*=\mathfrak{h}^\vee$ as well as a map $\psi^\vee: \mathfrak{h}^* \to \mathfrak{h}$. It is easy to check that
\[
    (\cdot, \cdot)_\mathfrak{h}^\vee =d^{-1} (\cdot, \cdot)_\mathfrak{h},\quad (\cdot, \cdot)_{\mathfrak{h}^*}^\vee=d (\cdot, \cdot)_{\mathfrak{h}^*}, \quad \psi^\vee = d \psi^{-1}.
\]
The fundamental weights $\omega_i\in \mathfrak{h}^*$ associated to the given simple coroots are defined by
\begin{equation}\label{weight}
		\langle\omega_i,\alpha_j^\vee\rangle=\delta_{ij}.
\end{equation}
The lattice generated by $\{\omega_i\}$ is the weight lattice of $\mathfrak{g}$, which we denote by $P$. By \eqref{rootcorootpairing} and \eqref{weight}, one has
\begin{equation}\label{linearcombofweights}
    (\alpha_1,\dots,\alpha_r)=(\omega_1,\dots,\omega_r)A, \qquad \text{{\em i.e.}~} \alpha_i=\sum_{j=1}^r a_{j i} \omega_j.
\end{equation}
Let $Q$ be the root lattice and $P^\vee=\Hom(Q,\mathbb{Z})\subset\mathfrak{h}$ be the dual lattice of $Q$ with dual basis $\{\omega_i^\vee\}$. Thus
\[
    (\alpha_i^\vee,\omega_j^\vee)=\langle \alpha_i^\vee,\psi(\omega_j^\vee) \rangle = d_j\delta_{ij},\quad (\alpha_1^\vee,\dots,\alpha_r^\vee)=(\omega_1^\vee,\dots,\omega_r^\vee)A^T.
\]
Let $Q^\vee=\Hom(P,\mathbb{Z})\subset\mathfrak{h}$ be the dual lattice of $P$, which is just the coroot lattice. 

Now let us recall the notion of character and cocharacter lattice. Let $\mathbb{G}_{\bf{m}}$ be the multiplicative group defined over $\mathbb{Q}$. Let $G$ be a semisimple algebraic group defined over $\mathbb{Q}$ with Lie algebra $\mathfrak{g}$. Let $H$ be the maximal torus of $G$ and $X^*(H)=\Hom(H,\mathbb{G}_{\bf{m}})$ the character lattice of $H$. For any $\gamma\in X^*(H)$, denote the multiplicative character by $\gamma\colon h\mapsto h^{\gamma}$. Let $X_*(H)= \Hom(\mathbb{G}_{\bf{m}},H)$ be the cocharacter lattice of $H$. Define the subset 
\[
    X_+^*=\{\lambda\in \mathfrak{h}^*\mid \langle \lambda,\alpha_i^\vee\rangle\in \mathbb{Z}_{\geqslant 0} \text{~for all~} i\in \bm{I}\}\subset X^*(H),
\]
which is the set of dominant weights of $G$.

In summary, we have the following lattices:
\[
	Q\subset X^*(H)\subset P; \quad Q^\vee\subset X_*(H)\subset P^\vee.
\]
\begin{Ex}\label{Ex:SL2}
	Let $G=\SL_2$ and $H$ be the subgroup of diagonal matrices. The roots of $\mathfrak{sl}_2$ give the following characters of $H$
	\[
		\alpha\colon
		\begin{bmatrix}
			a & 0\\
			0 & a^{-1}
		\end{bmatrix}\mapsto a^2,\
		-\alpha\colon
		\begin{bmatrix}
			a & 0\\
			0 & a^{-1}
		\end{bmatrix}\mapsto a^{-2}.
	\]
	Therefore $X^*(H)=\frac{1}{2}\mathbb{Z}\alpha=\mathbb{Z}\omega$, where $\omega=\frac{1}{2}\alpha$ is the only fundamental weight. The cocharacter lattice is $X_*(H)=\mathbb{Z}\alpha^\vee$, where $\alpha^\vee$ is the simple coroot of the root $\alpha$. The dual of the weight lattice is $\mathbb{Z}\alpha^\vee$. Thus we know:
	\[
		Q(\mathfrak{sl}_2)\subset X^*(H)= P(\mathfrak{sl}_2); \quad Q^\vee(\mathfrak{sl}_2)= X_*(H)\subset P^\vee(\mathfrak{sl}_2).
	\]
	We will come back to this example later.
\end{Ex}
The quadruple of $(X^*,Q;X_*,Q^\vee)$ is called \emph{root datum} of $G$, and the dual root datum $(X_*,Q^\vee;X^*,Q)$ is defined by switching characters with cocharacters, and roots with coroots. The Langlands dual group $G^\vee$ is the connected semisimple group whose root datum is dual to that of $G$. Let $H^\vee$ be the maximal torus of $G^\vee$. If $G$ is semisimple, the map $\psi$ restricts to cocharacter lattice:
\begin{Pro}\label{symmetrizer}
	One can choose the symmetrizer $\bm{d}$ such that the isomorphism $\psi$ restricts to a lattice (abelian group) homomorphism 
	\[
		\psi\colon X_*(H) \to X^*(H)=X_*(H^\vee),
	\]
	which induces a group homomorphism $\varPsi^H\colon H \to H^\vee$.
\end{Pro}

\begin{proof}
    Since $X_*(H)\subset P^\vee$ and $Q\subset X^*(H)$, it suffices to show that $\bm{d}$ can be chosen so that $\psi(P^\vee)\subset Q$. Considering \eqref{linearcombofweights} for the Lie algebra $\mathfrak{g}^\vee$ gives
    \begin{equation} \label{lincombfordual}
    (\omega^\vee_1,\dots,\omega_r^\vee) = (\alpha_1^\vee,\dots,\alpha_r^\vee) A^{-T},
    \end{equation}
    where we write $A^{-T}=(A^T)^{-1}$. Applying $\psi:\mathfrak{h}\to\mathfrak{h}^*$ to both sides of \eqref{lincombfordual}, one finds
    \begin{align*}
    (\psi(\omega^\vee_1),\dots,\psi(\omega^\vee_r)) = (\psi(\alpha_1^\vee),\dots,\psi(\alpha_r^\vee)) A^{-T} 
     = (\alpha_1,\dots,\alpha_r) DA^{-T}.
    \end{align*}
    It is enough then to choose $\bm{d}$ so that $DA^{-T}$ is an integer matrix; since $A$ is invertible over $\mathbb{Q}$, this is always possible.
\end{proof}

Note that if $G$ is simply connected, any symmetrizer $\bm{d}$ satisfies Proposition \ref{symmetrizer}. In the remainder of the paper, we fix a symmetrizer ${\bm d}$ as in Proposition \ref{symmetrizer}.

\begin{Ex}
	Here we list some examples of Langlands dual groups:
	\[
  		\SL(n)^{\vee}=\PSL(n), \ \SO(2n+1)^{\vee}=\Sp(2n), \ \Spin(2n)^{\vee}=\SO(2n)/\{\pm 1\},\ \SO(2n)^{\vee}=\SO(2n).
	\]	
\end{Ex}

\section{Positive Varieties}
\label{positivesection}

In this section, we briefly recall basic definitions of positivity theory and fix the notation.

Consider a split algebraic torus $S\cong \mathbb{G}_{\bf{m}}^n$. Denote the character lattice of $S$ by $S_t=\Hom(S, \mathbb{G}_{\bf{m}})$ and the cocharacter lattice by $S^t = \Hom(\mathbb{G}_{\bf{m}}, S)$. The lattices $S_t$ and $S^t$ are naturally in duality. The coordinate algebra $\mathbb{Q}[S]$ is the group algebra (over $\mathbb{Q}$) of the lattice $S_t$, that is, each $f \in \mathbb{Q}[S]$ can be written as
\begin{equation}\label{chicchi}
	f= \sum_{\chi \in S_t} c_\chi \chi,
\end{equation}
where only a finite number of coefficients $c_\chi$ are non-zero.

Let $\phi\colon S \to S'$ be a positive rational map (defined below). Following \cite{BKII}, we associate to $\phi$ a \emph{tropicalized map} $\phi^t\colon S^t\to (S')^t$ in the following way:

{\em Case 1.} If $\phi$ is a positive regular function on $S$, \emph{i.e.} $\phi$ has form as \eqref{chicchi} with all $c_\chi \geqslant 0$, then:
\[
	\phi^t\colon S^t \to \mathbb{G}_{\bf{m}}^t = \mathbb{Z}\ :\ \xi\mapsto \min_{\chi; \, c_\chi >0} \langle \chi, \xi \rangle,
\]
where $\langle \cdot, \cdot \rangle\colon S_t  \times S^t \to \mathbb{Z}$ is the canonical pairing.

{\em Case 2.} If $\phi$ is a positive rational function on $S$, \emph{i.e.} $\phi=f/g$ with $f,g$ positive regular functions:
\[
	\phi^t:=f^t - g^t.
\]

{\em Case 3.} Let $\phi\colon S \to S'$ be a \emph{positive rational map}, \emph{i.e.} the component functions of $\phi$ are positive rational functions on $S$. Define $\phi^t\colon S^t \to (S')^t$ as the unique map such that for every character $\chi \in S'_t$ and for every cocharacter $\xi \in S^t$ we have
\[
	\langle \chi, \phi^t(\xi) \rangle = (\chi \circ \phi)^t(\xi).
\]
A more concrete description is as follows. Let $\phi_1,\dots,\phi_m$ be the components of $\phi$ given by the splitting $S'\cong \mathbb{G}_{\bf{m}}^m$. Then, in the induced coordinates on $(S')^t,$ we have
\[
	\phi^t=(\phi_1^t,\dots,\phi_m^t).
\]
\begin{Ex}
	The positive rational map
	\begin{align*}
		\phi:\mathbb{G}_{\bf{m}}^3 \to \mathbb{G}_{\bf{m}}^3 \ :\ (x_1,x_2,x_3)& \mapsto \left(\frac{x_2x_3}{x_1+x_3},x_1+x_3,\frac{x_1x_2}{x_1+x_3} \right)
	\end{align*}
	has tropicalization
	\begin{align*}
		\phi^t\colon (\mathbb{G}_{\bf{m}}^3)^t\cong \mathbb{Z}^3 & \to (\mathbb{G}_{\bf{m}}^3)^t\cong \mathbb{Z}^3; \\
		(\xi_1,\xi_2,\xi_3) &\mapsto \left(\xi_2+\xi_3-\min\{\xi_1,\xi_3\},\min\{\xi_1,\xi_3\}, \xi_1+\xi_2-\min\{\xi_1,\xi_3\}\right).
	\end{align*}
	Note that $\phi^t$ is linear on the chambers $\xi_1<\xi_3$ and $\xi_1>\xi_3$.
\end{Ex}

\begin{Def}\cite[Definition 2.9, 2.16, 2.19]{ABHL}
	Let $X$ be an irreducible scheme over $\mathbb{Q}$. A \emph{toric chart} is an open embedding  $\theta\colon S \to X$ from a split algebraic torus $S$ to $X$. Two charts $\theta_1\colon S_1\to X$ and $\theta_2\colon S_2\to X$ are called \emph{positively equivalent} if $\theta_1^{-1}\circ\theta_2\colon S_2 \to S_1$ and $\theta_2^{-1}\circ\theta_1\colon S_1\to S_2$ are positive rational maps. A \emph{positive variety} is a pair $(X,\Theta_X)$, where $\Theta_X$ is a positive equivalence class of toric charts. If $\theta \in \Theta_X$, we sometimes write $\Theta_X=[\theta]$. Denote $(X,\theta)$ the \emph{framed positive variety} with a fixed toric chart $\theta$.
\end{Def}

Since $\theta$ is an open map, it induces an inclusion of coordinate algebras. We identify the coordinate algebra of $X$ with a subalgebra of $\mathbb{Q}[S]$. 

\begin{Ex} 
  Let $N\subset \SL_3$ be the group of unipotent upper-triangular matrices. Define $\theta\colon S=\mathbb{G}_{\bf{m}}^3\to N$ by
  \begin{align*}
    \theta(x_1,x_2,x_3) & = 
      \begin{bmatrix}
        1 & x_1 & 0 \\
        0 & 1 & 0 \\
        0 & 0 & 1
      \end{bmatrix}
      \begin{bmatrix}
        1 & 0 & 0 \\
        0 & 1 & x_2  \\
        0 & 0 & 1
      \end{bmatrix}
      \begin{bmatrix}
        1 & x_3 & 0 \\
        0 & 1 & 0 \\
        0 & 0 & 1
      \end{bmatrix}=
      \begin{bmatrix}
        1 & x_1+x_3 & x_1x_2 \\
        0 & 1 & x_2 \\
        0 & 0 & 1
      \end{bmatrix}.
  \end{align*}
  Then $\theta$ is a toric chart on $N$.
\end{Ex}

\begin{Def}
	A \emph{positive map} of positive varieties $\phi\colon(X,\Theta_X)\to (Y,\Theta_Y)$ is a rational map $\phi\colon X\to Y$ so that for some (equivalently any) $\theta_X\in \Theta_X$ and $\theta_Y\in \Theta_Y$, the rational map $\theta_Y^{-1} \circ\phi\circ \theta_X\colon S\to S'$ is positive.
\end{Def}
Tropicalization extends to positive varieties: if $(X,\theta\colon S\to X)$ is a framed positive variety, set
\[
	(X,\theta)^t:=\Hom(\mathbb{G}_{\bf{m}},S)=S^t.
\]
A positive rational map $\phi\colon (X,\theta_X)\to(Y,\theta_Y)$ has a tropicalization 
\[
	\phi^t:=(\theta_Y^{-1}\circ \phi\circ \theta_X)^t\colon (X,\theta_X)^t\to (Y,\theta_Y)^t,
\]
and tropicalization respects composition of positive rational maps. If $\theta,\theta'\colon \mathbb{G}_{\bf{m}}^k\to X$ are positively equivalent charts, the transition map $(\theta^{-1}\circ \theta')^t\colon\mathbb{Z}^k\to \mathbb{Z}^k$ is a piecewise $\mathbb{Z}$-linear bijection.

\begin{Def}
  Let $(X,\theta)$ be a framed positive variety. We distinguish a positive rational function $\Phi$, called a \emph{ potential} on $(X,\theta)$. The triple $(X, \theta, \Phi)$ is called a \emph{framed positive variety with potential}. We define similarly a \emph{positive variety with potential} as a triple $(X,[\theta],\Phi)$. For a framed positive variety with potential $(X,\theta, \Phi)$, we define the cone
  \[
    (X,\theta,\Phi)^t :=\left\{\xi\in (X,\theta)^t \mid \Phi^t(\xi)\geqslant 0\right\}\subset (X,\theta)^t,
  \]
  which we call the \emph{potential cone}. 
\end{Def}

\begin{Rmk}
  By imposing $0^t=-\infty$, we can consider $0$ as a potential. We then have $(X,\theta,0)^t=(X,\theta)^t$.
\end{Rmk}

\begin{Rmk}
  When $\Phi$ restricts to a regular function $\theta^*\Phi$ on the toric chart $\theta$, the cone $(X,\theta,\Phi)^t$ is convex. This type of situation occurs in all the examples we consider in this article.
\end{Rmk}
 
\section{Double Cluster Varieties}\label{Doubleseed}
In this section, we first recall some basic definitions from cluster theory. Then to each seed, we associate two cluster varieties, which are related by a collection of locally defined maps $\varPsi_\sigma$ indexed by seeds $\sigma$. These maps agree after tropicalization, giving a ``coordinate-independent'' comparison of the tropical varieties. The material of this chapter essentially follows \cite{FG2}.

\begin{Def}
	A \emph{seed} $\sigma=(I,J,M)$ consists of a finite set $I$, a subset $J\subset I$ and an integer matrix $M=\left[ M_{ij} \right]_{i,j\in I}$ which is skew-symmetrizable, \emph{i.e.} there exists a sequence of positive integers ${\bm d}=\{d_i\}_{i\in I}$ called a skew-symmetrizer such that $M_{ij}d_j=-M_{ji}d_i$. The \emph{principal part} of $M$ is given by $M_0=\left[ M_{ij} \right]_{i,j\in J}$.
\end{Def}
As with the symmetrizable matrix $A$ in section \ref{backgroundsection}, the existence of a skew-symmetrizer for $M$ easily implies that $M$ is skew-symmetrizable in the usual sense. Note that the submatrix $\widetilde{B}=[M_{ij}]_{i\in I,~j\in J}$ is called an {\em exchange matrix} and usually mutations of seeds are defined in terms of $\widetilde{B}$, however the seed matrix $M$ is more convenient for our purposes.

We associate a split algebraic torus to a given seed $\sigma$: 
\[
	\mathcal{A}_{\sigma}:=\mathbb{G}_{\bf{m}}^{|I|},
\]
and write $\{a_i\}_{i\in I}$ for the natural coordinates on $\mathcal{A}_\sigma$.

Recall that the matrix mutation of any matrix $M$ in direction $k$  is defined as:
\[
	\mu_k(M)_{ij}=\left\{
	\begin{aligned}
		\ &-M_{ij}, & &\text{~if~}\ k\in\{i,j\};\\
		\ &M_{ij}+\frac{1}{2}\Big(|M_{ik}|M_{kj}+M_{ik}|M_{kj}|\Big), & &\text{~otherwise~}.
	\end{aligned}\right.
\]
If $MD$ is skew-symmetric, one can easily show that $\mu_k(M)D$ is skew-symmetric as well. Define a \emph{mutation of a seed} $\sigma$ \emph{in direction} $k\in J$ as the seed $\sigma_k=(I_k,J_k,\mu_k(M))$, where $I_k=I, J_k=J$, together with a birational map of tori $\mu_k:\mathcal{A}_\sigma\to \mathcal{A}_{\sigma_k}$ given in terms of their coordinate algebras by:
\[
	\mu_k^*(a_i)=\left\{
	\begin{aligned}
		\ & a_i, & &\text{~if~}\ i\neq k;\\
		\ & a_k^{-1}\left(\prod_{M_{jk}>0}a_j^{M_{jk}}+\prod_{M_{jk}<0}a_j^{-M_{jk}}\right), & &\text{~if~}\ i=k.
	\end{aligned}\right.
\]
Two seeds will be called \emph{mutation equivalent} if they are related by a sequence of mutations. The equivalence class of a seed $\sigma$ is denoted by $|\sigma|$. 
\begin{Def}
	The cluster variety $\mathcal{A}\equiv\mathcal{A}_{|\sigma|}$ is the scheme obtained by gluing the $\mathcal{A}_{\sigma}$ for all $\sigma\in |\sigma|$ using the birational mutation maps. 
\end{Def}

\begin{Ex}\label{Ex:Stashef pentagon}
	(Stasheff pentagon)
	\begin{figure}[htb]
	\[
		\begin{tikzpicture}
			\newdimen\Rd \Rd=6cm
			\foreach \a in {5}{ 
				\node [dashed,regular polygon, regular polygon sides=\a, minimum size=\Rd, draw] at (\a*4,0) (A) {};
					\foreach \i in {1,...,\a}{
						\node [fill=white,regular polygon, regular polygon sides=\a, minimum size=\Rd/2] at (A.corner \i) {};
						\node (B) [fill=white, regular polygon, regular polygon sides=\a, minimum size=\Rd/3, draw] at (A.corner \i) {};
          				\foreach \i in {1,...,\a}{
            				\node [label=90+72*(\i-1):\i, inner sep=1pt] at (B.corner \i) {}; 
          				}
						\pgfmathsetmacro\x{int(Mod(2*(\i - 1),5) + 1)}
						\pgfmathsetmacro\p{int(Mod( (\x+2) , 5) + 1)}
						\pgfmathsetmacro\q{int(Mod( (\x+1) , 5) + 1)}
						\draw (B.corner \x) -- (B.corner \p);
						\draw (B.corner \x) -- (B.corner \q);
					}
			}
		\end{tikzpicture}
	\]
	\caption{Stasheff pentagon}
	\label{fig:Stashef pentagon}
	\end{figure}
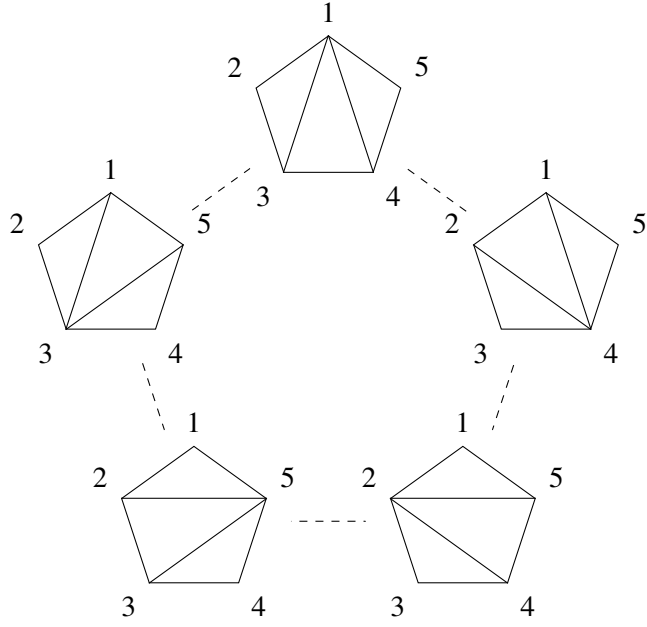
	To the very top pentagon in Figure \ref{fig:Stashef pentagon}, we associate a seed $(I,J,M)$, where $I=\{1,\dots,7\}$, $J=\{1,2\}$ and $M$ is given by
	\[
    M=
    \begin{bmatrix}
      M_{11}     & M_{12} \\[0.1cm]
      M_{12}^T & 0
    \end{bmatrix}, \text{~where~}\
		M_{11}=
		\begin{bmatrix}
			0 & -1\\
			1 &  0
		\end{bmatrix},
		M_{12}=
		\begin{bmatrix}
			1 & -1 &  1 & 0 &  0\\
			0 &  0 & -1 & 1 & -1
		\end{bmatrix}.
	\]
	To each edge $(m,n)$ in the top pentagon we associate a variable $b_{mn}$, which will be a coordinate on the seed torus $\mathcal{A}_{(I,J,M)}$. By ordering the variables $b_{mn}$ in the following way, we index them by $I$:
	\[
	    b_{13},b_{14},b_{12},b_{23},b_{34},b_{45},b_{15} .
	\]
	By the definition of mutation in direction $2$, we get
	\[
	    \mu_2(b_{14})=\frac{b_{13}b_{45}+b_{15}b_{34}}{b_{14}}.
	\]
	This mutation can be presented by the {\em Whitehead move} from edge $(1,4)$ to edge $(3,5)$ and the {\em Pl\"ucker relation}: Let $b_{35}:=\mu_{2}(b_{14})$ be the variable corresponding to edge $(3,5)$, then 
	\[
		b_{14}b_{35}=b_{13}b_{45}+b_{15}b_{34}.
	\]
	In fact, each dashed line in Figure \ref{fig:Stashef pentagon} is a Whitehead move and gives a cluster mutation. The algebra generated by $\{b_{ij}\}$ with all Pl\"ucker relations is the homogeneous coordinate ring of the Grassmannian $\mathcal{G}_2(5)$ of $2$-dimensional planes in the $5$-dimensional space.
	Note that the principal part of $M$ is $M_{11}$. More details can be found in  \cite{GSV}. 
\end{Ex}

\begin{Def}
Following \cite{FG}, we define the {\em (Langlands) dual seed} of $\sigma$ as $\sigma^\vee:=(I,J,-M^T)$. For the skew-symmetrizer ${\bm d}$ of $M$, fix an integer $d$ such that each $d_i$ divides $d$ for all $i\in I$. Then ${\bm d}^\vee:=\{d_i^\vee:=d/d_i\}$ is a skew-symmetrizer of $-M^T$. For a seed $\sigma$, denote the torus associated to the dual seed $\sigma^\vee$ by $\mathcal{A}^\vee_{\sigma}\equiv\mathcal{A}_{\sigma^\vee}$. 
\end{Def}

It is not hard to check that
\[
	\mu_k(-M^T)=-\mu_k(M)^T.
\]
In other words, we have $\mu_k(\sigma)^\vee=\mu_k(\sigma^\vee)$. Therefore, the tori $\mathcal{A}^\vee_\sigma$ assemble to a \emph{dual cluster variety} $\mathcal{A}^\vee$. That is, $\mathcal{A}^\vee= \mathcal{A}_{|\sigma^\vee|}= \mathcal{A}_{|(\sigma')^\vee|}$ for any $\sigma,\sigma'\in |\sigma|$. 

\begin{Def}
  The quadruple $(\mathcal{A},\mathcal{A}^\vee; {\bm d},d)$ is called a {\em double cluster variety}. We write $(\mathcal{A},\mathcal{A}^\vee)$ for short if  is the choice of ${\bm d}$ and $d$ is clear from context.
\end{Def}

%

Abusing notation, each seed $\sigma$ gives a toric chart $\sigma\colon \mathcal{A}_{\sigma}\hookrightarrow \mathcal{A}$, which will be called a \emph{cluster chart}. 
Mutation equivalent seeds $\sigma$ and $\sigma'$ give positively equivalent charts for $\mathcal{A}$. Denote $[\sigma]$ the class of positively equivalent charts given by the equivalence class $|\sigma|$ of the seed $\sigma$. By construction (and the same abuse of notation), the pair $(\mathcal{A},[\sigma])$ is a positive variety. Similarly, the maps $\sigma^\vee\colon\mathcal{A}_\sigma^\vee\hookrightarrow \mathcal{A}^\vee$ are toric charts on the positive variety $(\mathcal{A}^\vee,[\sigma^\vee])$.
Denote by $\mathbb{Q}[\mathcal{A}]$ the algebra of regular functions on $\mathcal{A}$. Then $\mathbb{Q}[\mathcal{A}]$ coincides with the \emph{upper cluster algebra} generated by the seed $\sigma$; see \cite{BFZ}. The algebra homomorphism $\sigma^*\colon \mathbb{Q}[{\mathcal A}]\to \mathbb{Q}[{\mathcal A}_\sigma]$ is an injection, for any cluster chart $\sigma$. 

Given a seed $\sigma$, there is a natural morphism of tori associated to the skew-symmetrizer $\bm{d}$:
\begin{equation}\label{com1}
	\varPsi_{\sigma}\colon \mathcal{A}_{\sigma}\to \mathcal{A}_\sigma^\vee \ : \ (x_{i_1},\dots,x_{i_{|I|}})\mapsto (x_{i_1}^{d_{i_1}},\dots,x_{i_{|I|}}^{d_{i_{|I|}}}).
\end{equation}
On the coordinate algebra, we have the algebra homomorphism
\[
	\varPsi_{\sigma}^*\colon \mathbb{Q}[\mathcal{A}_\sigma^\vee]  \to \mathbb{Q}[\mathcal{A}_\sigma] \ :\ a_i^\vee  \mapsto a_i^{d_{i}},\quad i\in I.
\] 
Since $\mu_k(M)$ is skew-symmetrized by $\bm{d}$ as well, for any $\sigma'\in |\sigma|$, there is another map of tori:
\[
	\varPsi_{\sigma'}\colon \mathcal{A}_{\sigma'}\to \mathcal{A}_{\sigma'}^\vee  \ : \ 
	 \ (x'_{i_1},\dots,x'_{i_{|I|}})\mapsto ({x'}_{i_1}^{d_{i_1}},\dots,{x'}_{i_{|I|}}^{d_{i_{|I|}}}).
\]
So for each seed $\sigma'\in |\sigma|$, there is a rational comparison map $\varPsi\colon \mathcal{A}_{|\sigma|} \to \mathcal{A}_{|\sigma|}$.

Note that $(\mathcal{A}^\vee,\mathcal{A}; {\bm d}^\vee,d)$ is also double cluster variety. Therefore, similar to \eqref{com1}, we have map $\varPsi_{\sigma^\vee}\colon \mathcal{A}_{\sigma}^\vee\to \mathcal{A}_\sigma$. Direct computation shows $\varPsi_{\sigma^\vee}\circ\varPsi_{\sigma}\colon \mathcal{A}_\sigma\to \mathcal{A}_\sigma$ is the map which simply raises each coordinate $a_i$ to the same power $d$. The cluster variety $\mathcal{A}$ and its dual $\mathcal{A}^\vee$ therefore play symmetric roles in the double cluster variety $(\mathcal{A},\mathcal{A}^\vee; {\bm d},d)$.

In what follows, write $\psi_\sigma:=\varPsi^t_\sigma$, where tropicalization is taken with respect to toric charts that are positively equivalent to $\sigma$ and $\sigma^\vee$. We will discuss the comparison map $\psi_{\sigma}=\varPsi_{\sigma}^t$ in more detail. Let us look at an example first.
\begin{Ex}
	We follow the notation in Example \ref{Ex:Stashef pentagon}. Since the matrix $M$ is skew-symmetric, the  dual of $\mathcal{G}_2(5)$ is itself by identifying $b_{ij}^\vee$ and $b_{ij}$. The skew-symmetrizer $\bm{d}$ can be chosen as $\diag(d,\dots,d)$ for $d\in \mathbb{Z}_+$. Then on each seed $\sigma$, we have:
    \[
		\varPsi_{\sigma}\colon \mathcal{A}_{\sigma}\to \mathcal{A}^\vee_{\sigma}\cong\mathcal{A}_{\sigma}\ \text{~s.t.~}\ \varPsi_{\sigma}^* (b_{ij}^\vee)=b_{ij}^d.
	\]
	So on the seed $\sigma$ containing edges $(1,3)$ and $(1,4)$, one computes:
	\begin{equation}\label{Ex:eq1}
		\varPsi_{\sigma}^*(b_{35}^\vee)=\frac{b_{13}^d b_{45}^d+b_{15}^d b_{34}^d}{b_{14}^d}
	\end{equation}
	On the seed $\sigma'$ containing edges $(3,1)$ and $(3,5)$, one has:
	\begin{equation}\label{Ex:eq2}
		\varPsi_{\sigma'}^*(b_{35}^\vee)=b_{35}^d=\left(\frac{b_{13}b_{45}+b_{15}b_{34}}{b_{14}}\right)^d.
	\end{equation}
	Note that right hand sides of \eqref{Ex:eq1} and \eqref{Ex:eq2} are equal after tropicalization:
	\[
	    \left( \frac{b_{13}^d b_{45}^d+b_{15}^d b_{34}^d}{b_{14}^d} \right)^t=\min \{ d\xi_{13}+d\xi_{45}, d\xi_{15}+d\xi_{34}\}-d\xi_{14}=\left( \frac{(b_{13}b_{45}+b_{15}b_{34})^d}{b_{14}^d} \right)^t,
	\]
	where $\xi_{mn}=b_{mn}^t$ is the tropicalization of $b_{mn}$.
\end{Ex}

Next we will generalize what happened in the previous example. Recall that if $\theta, \theta'\colon \mathbb{G}_{\bf{m}}^n \to X$ are positively equivalent charts on $X$, then $\id^t \colon (X,\theta)^t\to (X,\theta')^t$ is defined as $(\theta'\circ \theta^{-1})^t\colon  \mathbb{G}_{\bf{m}}^n\to  \mathbb{G}_{\bf{m}}^n$.

\begin{Pro}\label{Pro:compareCluster}
The tropical maps $\psi_\sigma$ agree for all $\sigma$. More precisely, let $\sigma$ be a seed, and $\mu$ be a sequence of mutations of $\sigma$. Then the following diagram commutes.
\[
\begin{tikzcd}
			(\mathcal{A},\sigma)^t \arrow[rr, "\id^t"] \arrow[dd, "\psi_\sigma"] && (\mathcal{A},\mu(\sigma))^t \arrow[dd, "\psi_{\mu(\sigma)}"] \\
			\\
			(\mathcal{A}^\vee,\sigma^\vee)^t \arrow[rr, "(\id^\vee)^t"] && (\mathcal{A}^\vee,\mu(\sigma^\vee))^t
		\end{tikzcd}
\]
Here we abbreviate $\id=\id_\mathcal{A}$ and $\id^\vee=\id_{\mathcal{A}^\vee}$.
\end{Pro}
\begin{proof}
	In fact, we only need to show the proposition for $\sigma $ and for $\mu=\mu_k$ a single mutation. Let $\sigma_k:=\mu_k(\sigma)$. Let $\{a_i\mid i\in I\}$ be the coordinates on $\mathcal{A}_{\sigma}$, and $\{a_{i'}\mid i'\in I'\}$  be the coordinates on $\mathcal{A}_{\sigma_k}$. And let $\{a_i^\vee\mid i\in I\}$ be the coordinates on $\mathcal{A}_{\sigma}^\vee$, and $\{a_{i'}^\vee\mid i'\in I'\}$ be the coordinates on $\mathcal{A}_{\sigma_k}^\vee$. On one hand, by definition:
	\[
		\varPsi_{\sigma_k}^*(a_{k'}^\vee)=a_{k'}^{d_k}=\mu_k^*(a_k)^{d_k}=a_k^{-d_k}\left(\prod_{M_{ik}>0}a_i^{M_{ik}}+\prod_{M_{ik}<0}a_i^{-M_{ik}}\right)^{d_k}.
	\]
	On the other hand, using the formula for mutation, we get:
	\begin{align*}
		\varPsi_{\sigma}^*(a_{k'}^\vee)=\varPsi_{\sigma}^*(\mu_k^*(a^\vee_k))&=a_k^{-d_k}\varPsi_{\sigma}^*\left(\prod_{M_{ki}>0}(a_i^\vee)^{M_{ki}}+\prod_{M_{ki}<0}(a_i^\vee)^{-M_{ki}}\right)\\
		&=a_k^{-d_k}\left(\prod_{M_{ki}>0}a_i^{d_iM_{ki}}+\prod_{M_{ki}<0}a_i^{-d_iM_{ki}}\right)\\
		&=a_k^{-d_k}\left(\prod_{M_{ik}<0}(a_i^{-M_{ik}})^{d_k}+\prod_{M_{ik}>0}(a_i^{M_{ik}})^{d_k}\right).
	\end{align*}
	Then the tropicalization gives
	\[
		(\varPsi_{\sigma_k}\circ\mu_k)^t\colon \left\{
    \begin{aligned}
		  \ &\xi_{k}^\vee \mapsto d_k\min \Big\{\sum_{M_{ik}<0} -M_{ik}\xi_i^\vee, \sum_{M_{ik}>0} M_{ik}\xi_i^\vee\Big\}-d_k\xi_{k'}^\vee;\ & &\\
		  \ &\xi_{i}^\vee \mapsto d_i\xi_{i'}, & &\text{~for~} i\neq k,
    \end{aligned}\right.
	\]
	and,
	\[
		(\mu_k\circ\varPsi_\sigma)^t\colon\left \{
    \begin{aligned}
		  \ &\xi_{k}^\vee \mapsto \min \Big\{\sum_{M_{ik}<0} - d_k M_{ik}\xi_i^\vee, \sum_{M_{ik}>0} d_k M_{ik}\xi_i^\vee\Big\}-d_k\xi_{k'}^\vee;\ & &\\
		  \ &\xi_{i}^\vee \mapsto d_i\xi_{i'}, & &\text{~for~} i\neq k.
    \end{aligned}\right.
	\]
	where $\{\xi_i\}_{i\in I}$ is the natural basis of $\mathcal{A}_\sigma^t= \Hom(\mathbb{G}_{\mathbf{m}},\mathcal{A}_\sigma)$, and similarly for ${\mathcal{A}^\vee_\sigma}^t$, $\mathcal{A}_{\sigma'}^t$, and ${\mathcal{A}_{\sigma'}^\vee}^t$. Thus
	\[
		(\mu \circ\varPsi_{\sigma})^t=(\varPsi_{\mu(\sigma)}\circ\mu)^t. \tag*{\qedhere}
	\] 
\end{proof}

 Note that our tropical map $\psi_\sigma$ is in general an injection (but not a bijection) of the lattice $\mathcal{A}_{\sigma}^t$ into the lattice ${\mathcal{A}^\vee_{\sigma}}^t$.

Given a split torus $H$ of rank $r$ and cluster variety $\mathcal{A}$ generated by a seed $\sigma=(I,J,M)$, denote $\widetilde{\mathcal{A}}=H\times \mathcal{A}$ the {\em extension} of $\mathcal{A}$ by $H$. Any choice of isomorphism of tori $H \cong\mathbb{G}_{\mathbf{m}}^r$ gives an isomorphism of $\widetilde{\mathcal{A}}$ and the cluster variety $\mathcal{A}_{|\widetilde{\sigma}|}$ generated by the seed
\[
  \widetilde{\sigma}:=\left(I\cup \{1,\dots,r\},J, \diag(M,0)\right).
\]
The variety $\widetilde{\mathcal{A}}$ will be called a {\em decorated} cluster variety, or cluster variety if the decoration $H$ is clear from the context. Note that $H=X_*(H)\otimes_{\mathbb{Z}} \mathbb{G}_\mathbf{m}$ and consider the group $H^\vee=X^*(H)\otimes_{\mathbb{Z}} \mathbb{G}_\mathbf{m}$. Then $H^\vee$ is the Langlands dual group of $H$ (in a slightly more general sense than was recalled in Section \ref{backgroundsection}). Define the (Langlands) dual of $\widetilde{\mathcal{A}}$ as $\widetilde{\mathcal{A}}^\vee:=H^\vee\times \mathcal{A}^\vee$. Given a double cluster variety $(\mathcal{A},\mathcal{A}^\vee, \mathbf{d},d)$, choose homomorphisms of tori $\varPsi^{H}\colon H\to H^\vee$ and $\varPsi^{H^\vee}\colon H^\vee\to H$ such that $\varPsi^{H^\vee}\circ\varPsi^{H}$ simply raises each coordinate to the $d$ power. Then the tuple $(\widetilde{\mathcal{A}}, \widetilde{\mathcal{A}}^\vee, \mathbf{d}, d, \varPsi^H, \varPsi^{H^\vee})$ is a {\em decorated double} cluster variety. We will often write $(\widetilde{\mathcal{A}}, \widetilde{\mathcal{A}}^\vee)$ for short.

On each seed of a decorated double cluster variety $(\widetilde{\mathcal{A}}, \widetilde{\mathcal{A}}^\vee)$, the comparison maps extends to:
\[
	\varPsi^{H}\times\varPsi_{\sigma}\colon \widetilde{\mathcal{A}}_{\sigma}\to \widetilde{\mathcal{A}}_{\sigma}^\vee  \quad \text{~and~}\quad \varPsi^{H^\vee}\times\varPsi_{\sigma^\vee}\colon \widetilde{\mathcal{A}}_{\sigma}^\vee \to \widetilde{\mathcal{A}}_{\sigma}.
\]
Let $\psi^{H}=(\varPsi^H)^t\colon (H)^t\to (H^\vee)^t$. By Proposition \ref{Pro:compareCluster}, the maps $\psi^H\times \psi_{\sigma}$ agree for all seeds $\sigma$.

\section{Double Bruhat Cells in \texorpdfstring{$G$}{G}  and \texorpdfstring{$G^\vee$}{Gvee} and Their Potentials}
\label{potentialsection}

In this section, we first recall the construction of a cluster algebra structure on the double Bruhat cell $G^{u,v}$ beginning with a double reduced word $\mathbf{i}$ for $(u,v)$. Then we show that $(G^{u,v}, G^{\vee;u,v})$ is a decorated double cluster variety with tropical comparison map $\psi_{\sigma(\mathbf{i})}$. Next we introduce ``factorization parameter'' toric charts $x_{\mathbf{i}}$ and $x_{\mathbf{i}}^\vee$ for $G^{u,v}$ and $G^{\vee;u,v}$, as well as corresponding comparison maps $\psi_{\mathbf{i}}$. We then focus on the double Bruhat cells $G^{w_0,e}$ and prove one of our main technical results, Theorem \ref{Main}. 
This says that the tropical comparison map
$\psi_{\sigma(\mathbf{i})}$ defined in terms of generalized minors agrees with the tropical comparison map $\psi_{\mathbf{i}}$ defined in terms of factorization parameters. Finally, we recall the definition of the {\em Berenstein-Kazhdan} (BK) potentials, the associated BK cones, and show that the comparison map $\psi_{\mathbf{i}}$ respects the crystal structure of the BK cones of $G$ and $G^\vee$.

\subsection{Backgroud on (Reduced) Double Bruhat Cells}

Let $G$ be a semisimple algebraic group as before. By fixing the simple roots $\alpha_i$ as in Section \ref{backgroundsection}, we have a maximal torus $H\subset G$ and a pair of opposite Borel subgroups $B, B_-$ of $G$ containing $H$. Denote by $U$ and $U_-$ the corresponding unipotent radicals of $B$ and $B_-$. Each triple $\alpha_i^\vee,E_i,F_i$ determines a group homomorphism $\phi_i\colon\SL_2\to G$ given by
\[
  \phi_i
    \begin{bmatrix}
      1 &  0\\
      a & 1
    \end{bmatrix}= \exp(a F_i) \subset U^-,~ 
  \phi_i
    \begin{bmatrix}
       1 & a \\
       0 & 1
    \end{bmatrix}= \exp(a E_i) \subset U,~
  \phi_i
    \begin{bmatrix}
      c & 0 \\
      0 & c^{-1}
    \end{bmatrix}=\alpha_i^\vee(c)\subset H
\]
for $a\in \mathbb{G}_{\bf{a}}$ and $c\in \mathbb{G}_{\bf{m}}$. Let $W=N(H)/H$ be the Weyl group of $G$ and $s_i\in W$ be the simple reflection generated by simple root $\alpha_i$. 
Let $w_0$ be the longest element in $W$ with 
\[
  m=\ell(w_0).
\]
The action of $W$ on $H$ gives rise to the action of $W$ on the character lattice $X^*(H)$, {\em i.e.} 
\[
  h^{w(\gamma)}=(w^{-1}hw)^{\gamma},\quad \gamma\in X^*(H),~h\in H.
\]
Using the $\SL_2$ homomorphisms $\{\phi_i\}$, define for $i \in \{1,\dots, r\}$,
\begin{equation}
\label{factorizationpars}
  \overline{s_i}:=\phi_i
    \begin{bmatrix} 
      0 & -1 \\
      1 & 0 
    \end{bmatrix},\ 
     x_{i}(t):=\phi_i
    \begin{bmatrix} 
      1 & t \\
      0 & 1 
    \end{bmatrix},\
      y_{i}(t):=\phi_i
    \begin{bmatrix} 
      1 & 0 \\
      t & 1 
    \end{bmatrix},\
     x_{-i}(t):=\phi_i
    \begin{bmatrix} 
      t^{-1} & 0 \\
      1 & t 
    \end{bmatrix}.\
\end{equation}
The $\overline{s}_i$'s satisfy the Coxeter relations of $W$, thus any decomposition of $w\in W$ into simple reflections gives the same lift $\overline{w}\in G$.

Let $g\mapsto g^\iota$ be the group antiautomorphism of $G$ given by
\[
	\alpha_i^\vee(c)^\iota = \alpha_i^\vee(-c), \quad x_i(t)^\iota = x_i(t), \quad y_i(t)^\iota = y_i(t),~i\in[1,r].
\]
Similarly, let $g\mapsto g^T$ be the group antiautomorphism of $G$ given by
\[
	\alpha_i^\vee(c)^T = \alpha_i^\vee(c), \quad x_i(t)^T = y_i(t), \quad y_i(t)^T = x_i(t),~i\in[1,r].
\]
Let $G_0 = U_- H U\subset G$ be the \emph{Gaussian decomposable} elements of $G$. For $g \in G_0$, write $g= [g]_- [g]_0 [g]_+$, where $[g]_-\in U_-$, $[g]_0\in H$, and $[g]_+\in U$. 

For a dominant weight $\mu\in X^*_+(H)$, the \emph{principal minor} $\Delta_\mu\in \mathbb{Q}[G]$ is uniquely determined by
\[
  \Delta_\mu(u_-au):=\mu(a), \text{~for any~} u_-\in U_-, a\in H, u\in U.
\]
For any two weights $\gamma$ and $\delta$ of the form $\gamma=w\mu$ and $\delta=v\mu$, where $w,v\in W$, the {\em generalized minor} $\Delta_{w\mu,v\mu}\in \mathbb{Q}[G]$ is given by
\[
  \Delta_{\gamma,\delta}(g)=\Delta_{w\mu,v\mu}(g):=\Delta_\mu(\overline w^{\,-1}g\overline v), \text{~for all~} g\in G.
\]
If $a_1,a_2\in H$ and $g\in G$, then
\begin{equation}
\label{BonusIdentity}
\Delta_{\gamma,\delta}(a_1 g a_2)= a_1^{\gamma} a_2^{\delta} \Delta_{\gamma,\delta}(g).
\end{equation}
If $G={\rm SL}_n$, the generalized minors are minors. 

For each pair of Weyl group elements $(u,v)$, a \emph{reduced} double Bruhat cell is defined by:
\[
  L^{u,v}:=U\overline{u}U\cap B_-vB_-.
\]
Note that multiplication in $G$ induces a biregular isomorphism $H\times L^{u,v}\cong G^{u,v}$. Let $\widehat{L}^{u,v}$ be the reduced double Bruhat cell of the universal cover $\widehat{G}$ of $G$ and let
\[
  \mathcal{p}\colon \widehat{G}\to G
\]
be the covering map. The cell $\widehat{L}^{u,v}$ can be characterized by the following 
\begin{Pro}\cite[Proposition 4.3]{BZ01}\label{charaofz}
  An element $x\in \widehat{G}^{u,v}$ belongs to $\widehat{L}^{u,v}$ if and only if 
  \[
    \Delta_{u\omega_i,\omega_i}(x)=1, \quad \forall i\in \bm{I}.
  \] 
\end{Pro}
\begin{Cor}
The restriction of $\mathcal{p}$ to $\widehat{L}^{u,v}$ is an isomorphism $\widehat{L}^{u,v}\to L^{u,v}$.
\end{Cor}
\begin{proof}
  For $h\in \widehat{H}$, we know $h=\id$ if and only if $h^{\omega_i}=1$ for all $i\in \bm{I}$. Let $x\in L^{u,v}$ and consider some $\widehat{x}, \widehat{x}'\in \mathcal{p}^{-1}(x)\subset \widehat{L}^{u,v}$. Then $\widehat{x}'=\widehat{x}h$ for some $h\in \widehat{H}$. By Proposition \ref{charaofz} and \eqref{BonusIdentity} we have $h^{\omega_i}=1$ for all $i$, which implies there is a unique lift of $x$.
\end{proof}
Therefore the generalized minors $\Delta_{u\omega_i,v\omega_i}$ can be viewed as well defined functions on $L^{u,v}$ under the isomorphism $\mathcal{p}$. By abuse of notation, we write $\Delta_{u\omega_i,v\omega_i}(z)$ for $z\in L^{u,v}$ instead of $\Delta_{u\omega_i,v\omega_i}(\mathcal{p}^{-1}z)$.

\subsection{Double Cluster Algebras on Double Bruhat Cells}\label{dualondoublecell}

In this section, we recall how to make $G^{u,v}$ into a cluster variety, for any pair $(u,v) \in W\times W$. We will begin by working with $L^{u,v}$. After decomposing $G^{u,v}=H\times L^{u,v}$, we will get a decorated cluster variety $G^{u,v}$ by extending $L^{u,v}$ to $H\times L^{u,v}$.

A \emph{double reduced word} $\mathbf{i}=(i_1,\dots,i_n)$ for $(u,v)$ is a shuffle of a reduced word for $u$, written in the alphabet $\{-r,\dots,-1\}$, and a reduced word for $v$, written in the alphabet $\{1,\dots,r\}$, where $n=\ell(u)+\ell(v)$. For $k\in [1,n]$, we denote by
\begin{equation}\label{k+}
    k^+=\min\{j\mid j>k,|i_j|=|i_k|\}.
\end{equation}
If $|i_j|\neq |i_k|$ for all $j>k$, we set $k^+=n+1$. An index $k$ is \emph{$\mathbf{i}$-exchangeable} if $k^+\in [1,m]$. Let $\bm{e}(\mathbf{i})$ denote the set of all $\mathbf{i}$-exchangeable indices.
\begin{Rmk}\label{aboutsign}
	The minus signs on the letters of the reduced word for $u$ are occasionally troublesome, so let us make the following abbreviations. For $i,j\in [-r,-1]\cup[1,r]$, let
	\[
		d_i = d_{|i|},~a_{i,j} = a_{|i|,|j|},~\omega_i = \omega_{|i|}, s_{i}=s_{|i|}
	\]
	extending the notation for the skew-symmetrizer, Cartan matrix, and fundamental weights, and simple reflection respectively. Our notation is set up in this way to agree with that of \cite{BFZ}.
\end{Rmk}

Fix a double reduced word $\mathbf{i}$ of $(u,v)$. We will define a seed $\sigma(\mathbf{i}):=(I,J,M(\mathbf{i}))$ as follows. Let $I=[-r,-1]\cup{\bm e}(\mathbf{i})$ and $J=\bm{e}(\mathbf{i})$. The $n\times n$ matrix $M(\mathbf{i})$ for the seed $\sigma(\mathbf{i})$ is constructed as follows. For $k,l\in I$, set $p=\max\{k,l\}$ and $q=\min\{k^+,l^+\}$, and let $\epsilon(k)$ be the sign of $k$. Then by \cite[Remark 2.4]{BFZ}
\begin{equation}\label{matixinseed}
  \begin{split}
    M(\mathbf{i})_{kl}=\left\{
    \begin{aligned}
      &-\epsilon(k-l)\cdot\epsilon(i_p),& &\text{~if~} p=q;\\
      &-\epsilon(k-l)\cdot\epsilon(i_p)\cdot a_{i_k,i_l},& &\text{~if~} p<q \text{~and~} \epsilon(i_p)\epsilon(i_q)(k-l)(k^+-l^+)>0;\\
      &0, & &\text{~otherwise~}.
    \end{aligned}\right.
  \end{split}
\end{equation}
Here, recall that $A=[a_{ij}]$ is the symmetrizable Cartan matrix of $\mathfrak{g}$. Let 
\begin{equation}
\label{ddef}
	\bm{d}_{\mathbf{i}}=\{d_{i_1},\dots,d_{i_m}\},
\end{equation}
where the sequence $\bm{d}=\{d_1,\dots,d_r\}$ is the fixed symmetrizer of $A$. It is easy to see that $\bm{d}_{\mathbf{i}}$ is a skew-symmetrizer of $M(\mathbf{i})$. As before, we fix a positive integer $d$ such that each $d_i$ divides $d$.

For a double reduced word $\mathbf{i}$ of $(u,v)$ and $k\in [1,n]$, denote by
\[
	u_k=\prod_{\substack{l=1,\ldots,k\\ i_l<0}} s_{i_l}, \quad v_k=\prod_{\substack{l=n,\ldots,k+1\\ i_l>0}} s_{i_l},
\]
where the index is increasing in the product on the left, and decreasing in the product on the right.
Define the generalized minors 
\begin{equation}\label{fullcluster}
	\Delta_k:=\Delta_{u_k\omega_{i_k},v_k\omega_{i_k}}\text{~for~}k\in [1,n]; \quad \Delta_k:=\Delta_{\omega_k,v^{-1}\omega_k}\text{~for~}k\in[-r,-1].	
\end{equation}
\begin{Thm}\cite[Theorem 2.10]{BFZ}\label{Thm2.10}
  For every double reduced word $\mathbf{i}$ for $(u,v)$, let $\mathcal{A}_{|\sigma(\mathbf{i})|}$ be the cluster variety generated by the seed $\sigma(\mathbf{i})$. Then the map given by
  \[
    \varphi^*\colon \mathbb{Q}[\mathcal{A}_{|\sigma(\mathbf{i})|}]\to \mathbb{Q}[L^{u,v}] \ :\ a_k\mapsto \Delta_k, \ \text{~for~}\  k \in I=[-r,-1]\cup{\bm e}(\mathbf{i})
  \]
  is an isomorphism of algebras.
\end{Thm}
\begin{Rmk}
  By \cite[Eq (2.11)]{BFZ}, the set of cluster variables on the chart $\sigma(\mathbf{i})$ of the double Bruhat cell $\widehat{G}^{u,v}$ for simply connected $\widehat{G}$ is $\{\Delta_k\mid k\in [-r,-1]\cup [1,n]\}$. Recall that  for $x\in \widehat{L}^{u,v}\subset \widehat{G}^{u,v}$, we have $\Delta_{u\omega_i,\omega_i}(x)=1$.  Thus Theorem \ref{Thm2.10} follows from \cite[Theorem 2.10]{BFZ} by applying $\Delta_{u\omega_i,\omega_i}=1$ and identifying $\mathbb{Q}[L^{u,v}]$ and $\mathbb{Q}[\widehat{L}^{u,v}]$.
\end{Rmk}

Since the Weyl groups of $G$ and $G^\vee$ are isomorphic, the reduced word $\mathbf{i}$ also gives the reduced double Bruhat cell $L^{\vee;u,v}$ for $G^\vee$ the structure of a cluster variety. Moreover, we have:
\begin{Cor}
   Fix $(u,v)\in W\times W$. Let ${\bm d}_\mathbf{i}$ be as in \eqref{ddef} and let $d$ be the integer fixed in Chapter \ref{backgroundsection}. Then the quadruple $(L^{u,v},L^{\vee;u,v}; {\bm d}_{\mathbf{i}},d)$ is a double cluster variety.
\end{Cor}
\begin{proof}
  What we need to show actually is $(L^{u,v})^\vee\cong L^{\vee;u,v}$, where $(L^{u,v})^\vee$ is the dual cluster variety of $L^{u,v}$. Let $\left(I,J,M^\vee(\mathbf{i})\right)$ be the initial seed of $L^{\vee;w_0,e}$. Following the definitions, one obtains
  \[
    \left(I,J, M(\mathbf{i})\right)^\vee=\left( I,J,-M(\mathbf{i})^T\right)=\left(I,J,M^\vee(\mathbf{i})\right). \tag*{\qedhere}
  \]
\end{proof}

For a seed $\sigma\in |\sigma(\mathbf{i})|$, denote by $\Psi_\sigma^L\colon L^{u,v}\to L^{\vee;u,v}$ the comparison map for the double cluster variety $(L^{u,v},L^{\vee;u,v})$. Extending the cluster variety $L^{u,v}$ by $H$, we get the decorated cluster variety $G^{u,v}\cong H\times L^{u,v}$. For any seed $\sigma$ on $L^{u,v}$, the following map gives a positive chart on $G^{u,v}$:
\begin{equation}\label{clusterchart}
   \id\times \sigma\colon H\times \mathbb{G}_{\mathbf{m}}^n\to H\times L^{u,v}=G^{u,v},
\end{equation}
which will be denoted by $\sigma$ as well if there is no ambiguity. Combining with $\varPsi^H\colon H \to H^\vee$ as in Proposition \ref{symmetrizer}, we have the following comparison map on the decorated cluster variety:
\begin{equation}\label{clustercompare}
  \varPsi_\sigma:=\varPsi^H\times \varPsi_\sigma^L\colon G^{u,v}= H\times L^{w_0,e}\to G^{\vee;u,v}=H^\vee \times L^{\vee;u,v}.
\end{equation}
The tuple $(G^{u,v}, G^{\vee; u,v}, { \bm d}_\mathbf{i}, d, \varPsi^H, \varPsi^{H^\vee})$ is then a decorated double cluster variety.

\subsection{Comparison map in factorization parameters}\label{dualondouble}

In this section we describe certain toric charts on double Bruhat cells which are positively equivalent to the ones already considered. We introduce a comparison map $\psi_\mathbf{i}$. In the next section we show this coincides (tropically) with the comparison map of Section \ref{dualondoublecell}.

Let $G$ be a semisimple algebraic group as before and $(u,v)$ a pair of elements in $W$. Recall we denote $n=\ell(u)+\ell(v)$. By \cite[Proposition 4.5]{BZ01}, the map
\begin{equation}\label{zcoord}
	x_\mathbf{i}\colon \mathbb{G}_{\bf{m}}^n \xrightarrow{\ \sim\ } \widehat{L}^{u,v} \xrightarrow{\ \mathcal{p}\ }L^{u,v}\ :\ 
	(t_1,\dots,t_n)\mapsto x_{i_1}(t_1)\cdots x_{i_n}(t_n),
\end{equation}
gives a toric chart on $L^{u,v}$. Thus factoring $G^{u,v}$ as $H\times L^{u,v}$ gives the toric chart:
\[
	x_{\mathbf{i}}\colon H\times \mathbb{G}_{\bf{m}}^n \xrightarrow{\ \sim\ } G^{u,v}=H\times L^{u,v}\ :\ 
	(h, t_1,\dots,t_n)\mapsto hx_{i_1}(t_1)\cdots x_{i_n}(t_n).
\]
We have overloaded the notation $x_\mathbf{i}$ here but the meaning will be clear from context.

As shown in \cite{BZ01}, if $\mathbf{i}$ is a double reduced word for $(u,v)$, then the factorization chart $x_{\mathbf{i}}$ on $G^{u,v}$ is positively equivalent to the cluster chart $\sigma(\mathbf{i})$. Moreover, if $\mathbf{i}$ and $\mathbf{i}'$ are both double reduced words for $(u,v)$, then the toric chart $x_\mathbf{i}$ is positively equivalent to $x_{\mathbf{i}'}$. If $x_{\mathbf{i}}(h,t_1,\dots, t_m) = x_{\mathbf{i}'} ( h,p_1,\dots, p_m)$, then the coordinates $p_j$'s can be expressed as rational functions of the $t_j$'s, via a series of positive equivalences \emph{moves}. These are given explicitly in Propositions 7.1, 7.2, and 7.3 of \cite{BZ01}, where they are called ``(mixed) $d$-moves''. 

For a double reduced word $\mathbf{i}$ of $(u,v)$, let $\varPsi_\mathbf{i} \colon G^{u,v} \to G^{\vee;u,v}$ be the positive rational map which is given in terms of the toric charts $x_\mathbf{i}$ and $x_\mathbf{i}^\vee$ by
 \begin{equation}\label{xtoxvee}
	x_\mathbf{i}(h, t_1,\ldots,t_n) \mapsto x_\mathbf{i}^\vee(\varPsi^H(h), t_1^{d_{i_1}},\ldots, t_n^{d_{i_n}}).
\end{equation}
We will write $\psi_\mathbf{i}=\varPsi_\mathbf{i}^t$ for the tropicalized comparison map. 
The comparison maps $\psi_\mathbf{i}$ for different double reduced words $\mathbf{i}$ all agree after tropicalization. 

\begin{Pro}\label{dmovesok}
	Let $\mathbf{i}$ and $\mathbf{i}'$ be double reduced words for $(u,v)$. Then the following diagram commutes,
	\[
		\begin{tikzcd}
			(G^{u,v},x_{\mathbf{i}})^t \arrow[rr, "\id^t"] \arrow[dd, "\psi_{\mathbf{i}}"] && (G^{u,v},x_{\mathbf{i}'})^t \arrow[dd, "\psi_{\mathbf{i}'}"] \\
			\\
			(G^{\vee;u,v},x^\vee_{\mathbf{i}})^t \arrow[rr, "(\id^\vee)^t"] && (G^{\vee;u,v},x^\vee_{\mathbf{i}'})^t 
		\end{tikzcd}
	\]
	where we recall that $\id^t=\id_{G^{u,v}}^t=(x_{\mathbf{i}'}^{-1}\circ x_{\mathbf{i}})^t$ by definition, and we abbreviate $\id^\vee= \id_{G^{\vee;u,v}}$.
\end{Pro}
\begin{proof}
By the previous discussion it is enough to assume that $\mathbf{i}$ and $\mathbf{i}'$ are related by a single move. Then the proposition follows by direct computation; we will give the proof for one type of move and leave the rest to the reader.

Say $i,j\in \{-r,\dots, -1\}$ with $a_{i,j} = -1$ and $a_{j,i} = -2$. Without loss of generality assume $d_{i} = 1$ and $d_{j} = 2$. Let 
\[
	\mathbf{i}  = (i_1,\dots, i_k,  i,j,i,j, i_{k+5}, \dots, i_n),\quad
	\mathbf{i'}  = (i_1,\dots, i_k, j,i,j,i , i_{k+5}, \dots, i_n).
\]
Then by \cite[Proposition 7.3]{BZ01},
\begin{align*}
	\varPsi_{\mathbf{i}'}\circ x_{\mathbf{i}'}^{-1}\circ x_{\mathbf{i}}(h, t_1,\dots,  t_m)  & = \varPsi_{\mathbf{i'}}( h, p_1,\dots, p_m)\\
	& = (\varPsi^H(h), p_1^{d_{i_1}}, \dots, p_{k+1}^2,p_{k+2} ,p_{k+3}^2 ,p_{k+4} , \dots, p_m^{d_{i_m}}),
\end{align*}
where 
\begin{align*}
	p_{k+1} &=\left( \frac{t_{k+1}}{t_{k+2}} + \frac{t_{k+2}}{t_{k+3}} + \frac{1}{t_{k+4}}\right)^{-1}\  &
	p_{k+2} & = \left( \frac{1}{t_{k+1}}\left( \frac{t_{k+2}}{t_{k+3}} + \frac{1}{t_{k+4}}\right)^2 + \frac{1}{t_{k+3}}\right)^{-1}\\
	p_{k+3} & = t_{k+2} + t_{k+1} t_{k+4} + \frac{t_{k+2}^2 t_{k+4}}{t_{k+3}}\  &
	p_{k+4} &= t_{k+1} + t_{k+3}\left(\frac{t_{k+2}}{t_{k+3}} + \frac{1}{t_{k+4}}\right)^2
\end{align*}
and $p_i = t_i$ otherwise. On the other hand, again using \cite[Proposition 7.3]{BZ01} one finds
\begin{align*}
	(x_{\mathbf{i}'}^\vee)^{-1}\circ x_{\mathbf{i}}^\vee \circ \varPsi_{\mathbf{i}}(h, t_1,\dots,  t_m) & =(x_{\mathbf{i}'}^\vee)^{-1}\circ x_{\mathbf{i}}^\vee(\varPsi^H(h), t_1^{d_{i_1}}, \dots, t_m^{d_{i_m}}) \\
	& = (\varPsi^H(h),P_1,\dots, P_m),
\end{align*}
where
\begin{align*}
	P_{k+1} & = \left( \frac{1}{t_{k+4}^2} +\frac{1}{t_{k+2}^2}\left(\frac{t_{k+2}^2}{t_{k+3}} +t_{k+1}\right)^2\right)^{-1}\  &
	P_{k+2} & = \left( \frac{1}{t_{k+3}} + \frac{1}{t_{k+4}^2 t_{k+1}} + \frac{t_{k+2}^2}{t_{k+3}^2 t_{k+1}}\right)^{-1}\\
	P_{k+3} & = t_{k+4}^2\left(\frac{t_{k+2}^2}{t_{k+3}} + t_{k+1}\right)^2 +t_{k+2}^2\  &
	P_{k+4} & = \frac{t_{k+3}}{t_{k+4}^2} + \frac{t_{k+2}^2}{t_{k+3}} + t_{k+1}
\end{align*}
and $P_i = t_i^{d_{i}}$ otherwise. Then it is easy to verify that $\varPsi_\mathbf{i}$ and $\varPsi_\mathbf{i'}$ agree after tropicalization.

The proofs for the other types of move (described in Propositions 7.1, 7.2, and 7.3 of \cite{BZ01}) are along the same lines. The computation for the two types of moves associated to type $G_2$ are slightly more involved. One must show that some terms in the expressions for the $p_i$'s and $P_i$'s do not contribute after tropicalization; this can be done using the fact that the tropicalization of $(A + B)^k$ is the same as the tropicalization of $A^k + B^k$, for positive functions $A$ and $B$ and positive integers $k$. This verification is straightforward and tedious.
\end{proof}

\subsection{Compatibility of Positive Structures}

In this section and the remainder of the paper, we will focus on the variety $G^{w_0,e}$. We now prove one of the main results of this article: If $\mathbf{i}=(i_1,\ldots,i_m)$ is a double reduced word for $(w_0, e)$, then under the tropical change of coordinates between $(G^{w_0,e},x_{\mathbf{i}})^t$ and $(G^{w_0,e},\sigma)^t$, the map $\psi_\mathbf{i}:=\varPsi^t_{\mathbf{i}}$ agrees with $\psi_\sigma:=\varPsi^t_\sigma$, where $\varPsi_\mathbf{i}$ is defined in \eqref{xtoxvee} and $\varPsi_\sigma$ is defined in \eqref{clustercompare}. 

\begin{Thm}\label{Main}
    For a double reduced word $\mathbf{i}$ of $(w_0,e)$, and let $x_\mathbf{i}$ and $\sigma=\sigma(\mathbf{i})$ be the factorization chart and cluster charts for $G^{w_0,e}$, respectively. Then the following diagram commutes,
    \[
    	\begin{tikzcd}
			(G^{w_0,e}, x_\mathbf{i})^t \arrow[rr, "\id^t"] \arrow[dd, "\psi_\mathbf{i}"] && (G^{w_0,e}, \sigma)^t \arrow[dd, "\psi_\sigma"] \\
			\\
			(G^{\vee; w_0,e}, x_\mathbf{i}^\vee)^t \arrow[rr, "(\id^\vee)^t"] && (G^{\vee; w_0,e}, \sigma^\vee)^t
		\end{tikzcd}
 	\]
 	where we recall that $\id^t=(\sigma^{-1}\circ x_{\mathbf{i}})^t$ by definition.
\end{Thm}

Before the proof of Theorem \ref{Main}, let us develop some preliminary results. Throughout this section, we use the convention in Remark \ref{aboutsign}. 

For a pair of Weyl group elements $(u,v)$ with $\ell(u)=p$ and $\ell(v)=q$, a double reduced word $\mathbf{i}=(i_1,\ldots,i_{p}, j_1,\ldots,j_{q})$ is {\em separated} if $i_1,\dots,i_{p}\in [-r,-1]$ and $j_1,\dots,j_{q}\in[1,r]$, and if also $\ell(us_{j_{q}})=\ell(u)+1$. For a separated double reduced word $\mathbf{i}$ for $(u,v)$,
define
\[
	\hat{\mathbf{i}}:=(i_1,\ldots,i_{p}, j_1,\ldots,j_{q-1},-j_{q}); \quad \tilde{\mathbf{i}}:=(i_1,\ldots,i_{p},-j_q, j_1,\ldots,j_{q-1}).
\]
Note that both $\hat{\mathbf{i}}$ and $\tilde{\mathbf{i}}$ are double reduced words for $(us_{j_q},vs_{j_q})$.
We then define the following birational map in terms of the toric charts $x_{\mathbf{i}}, x_{\hat{\mathbf{i}}}$:
\[
	P_{\mathbf{i}}\colon  L^{u,v}  \to L^{us_{j_q},vs_{j_q}} \ :\  
	x_{\mathbf{i}}(t_1,\dots, t_n)  \mapsto x_{\hat{\mathbf{i}}}(t_1,\dots,t_n^{-1}).
\]
We denote the analogous map for $L^{\vee;u,v}$ by $P_\mathbf{i}^\vee$.
\begin{Lem}\label{lem5.10}
	Let $\mathbf{i}=(i_1,\ldots,i_{p}, j_1,\ldots,j_{q})$ be a separated double reduced word for $(u,v)$. Then the following diagram commutes.
	\[
	\begin{tikzcd}
		(L^{u,v}, x_\mathbf{i})^t \arrow[rr, "P_{\mathbf{i}}^t"] \arrow[dd, "\psi_\mathbf{i}"] &&(L^{us_{j_q},vs_{j_q}}, x_{\hat{\mathbf{i}}})^t \arrow[rr, "\id^t"] \arrow[dd, "\psi_{\hat{\mathbf{i}}}"] && (L^{us_{j_q},vs_{j_q}}, x_{\tilde{\mathbf{i}}})^t \arrow[dd, "\psi_{\tilde{\mathbf{i}}}"] \\
		\\
		(L^{\vee;u,v}, x^\vee_\mathbf{i})^t \arrow[rr, "(P_{\mathbf{i}}^\vee)^t"]  &&(L^{\vee;us_{j_q},vs_{j_q}}, x_{\hat{\mathbf{i}}})^t \arrow[rr, "(\id^\vee)^t"]  && (L^{\vee;us_{j_q},vs_{j_q}}, x^\vee_{\tilde{\mathbf{i}}})^t
	\end{tikzcd}
	\]
\end{Lem}
\begin{proof}
	The square on left commutes directly from the definition of $P_\mathbf{i}$ and $\psi_{\mathbf{i}}$. The commutativity of the right square is just Proposition \ref{dmovesok} for ``mixed'' moves.
\end{proof}

Now, let $\mathbf{i}=(i_1,\ldots,i_n)$ be a double reduced word for $(e,v)$. For $k\in [0,n]$, let 
\[
\mathbf{i}_k:=(-i_n,\ldots,-i_{k+1},i_1,\ldots,i_{k}).
\]
Note that $\mathbf{i}_k$ and $\tilde{\mathbf{i}}_k= \mathbf{i}_{k-1}$ are both separated. Then $\mathbf{i}_n=\mathbf{i}$ and $\mathbf{i}_0$ is $-\mathbf{i}$, written in the opposite order.

\begin{Lem}\label{lem5.11}
Fix a double reduced word $\mathbf{i}=(i_1,\dots,i_n)$ for $(e,v)$. Let $x\in L^{e,v}$ be in the image of the toric chart $x_\mathbf{i}$. Then
\[
[x \overline{v}]_- [x\overline{v}]_0 = P_{\mathbf{i}_{1}}\circ P_{\mathbf{i}_{2}}\circ\cdots \circ P_{\mathbf{i}_n}(x)\in L^{v^{-1},e}.
\]
\end{Lem}
\begin{proof}
	For $k\in[0,n]$, write $v_k=s_{i_{n-k}}\cdots s_{i_2} s_{i_1}$. For $i>0$, it is easy to check that $x_{i}(t)\overline{s_{i}} = x_{-i}(t^{-1}) x_{i}(-t^{-1})$, then
	\begin{align*}
		x_{\mathbf{i}_n}(t_1,\ldots,t_n)\overline{v}_0& = x_{i_1}(t_1) \cdots x_{i_n}(t_n) \overline{s_{i_n}} \cdot \overline{v}_1\\
		&=x_{i_1}(t_1) \cdots x_{i_{n-1}}(t_{n-1}) x_{-i_n}(t_n^{-1}) x_{i_n}(-t_n^{-1}) \cdot \overline{v}_1\\
		&=\left(P_{\mathbf{i}_n}(x_{\mathbf{i}_n}(t_1,\ldots,t_n))\overline{v}_1\right)\cdot \left(\overline{v}_1^{-1}x_{i_n}(-t_n^{-1}) \overline{v}_1\right).
	\end{align*}
	In the last line, note that $\overline{v}_1^{-1}x_{i_n}(-t_n^{-1}) \overline{v}_1\in U$; this follows from well-known results on the Weyl group, as in Section 10.2 of \cite{Humph}. By writing $P_{\mathbf{i}_n}(x_{\mathbf{i}_n}(t_1,\dots,t_n)) = x_{\mathbf{i}_{m-1}}(t'_1,\dots,t_n')$, we can repeat the argument from above:
	\begin{align*}
	x_{\mathbf{i}_{n-1}}(t_1',\ldots,t_n')\overline{v}_1 & = x_{-i_n}(t_1') x_{i_1}(t_2')\cdots x_{i_{n-1}}(t_n') \overline{s}_{i_{n-1}} \overline{v}_2 \\
	& = (P_{\mathbf{i}_{n-1}}(x_{\mathbf{i}_{n-1}}(t_1',\dots,t_n')) \overline{v}_2)\cdot (\overline{v}_2^{-1} x_{i_{n-1}}(-{t_n'}^{-1}) \overline{v}_2).
	\end{align*}
	As before, $\overline{v}_2^{-1} x_{i_{n-1}}(-{t_n'}^{-1}) \overline{v}_2\in U$. Repeating the argument $n$ times and taking $[\cdot]_- [\cdot]_0$ on both sides, we get the desired formula.
\end{proof}

Now, fix a double reduced word $\mathbf{i}$ for $(w_0,e)$ as in the statement of Theorem \ref{Main}. Consider the biregular \emph{twist map} from \cite{BZ01}
\begin{equation}\label{twist}
	\zeta\colon  L^{e,w_0} \to L^{w_0,e} \ :\ x\mapsto ([x\overline{w_0}]_{-} [x\overline{w_0}]_{0})^\iota,
\end{equation}
which is a biregular map, and also an isomorphism of the positive varieties $(L^{e,w_0},x_{\mathbf{-i}})$ and $(L^{w_0,e}, x_{\mathbf{i}})$.

\begin{proof}[Proof of Theorem \ref{Main}]
	Consider the following diagram.
	\begin{equation}\label{goaldiag}
    	\begin{tikzcd}
			(L^{e,w_0}, x_\mathbf{-i})^t \arrow[rr, "\zeta^t"]  \arrow[dd, "\psi_\mathbf{-i}"] &&(L^{w_0,e}, x_{\mathbf{i}})^t \arrow[rr, "\id^t"]  \arrow[dd, "\psi_{\mathbf{i}}"] && (L^{w_0,e}, \sigma)^t \arrow[dd, "\psi_{\sigma}"] \\
			\\
			(L^{\vee;e,w_0}, x_\mathbf{-i}^\vee)^t \arrow[rr, "(\zeta^\vee)^t"]   &&(L^{\vee;w_0,e}, x_{\mathbf{i}}^\vee)^t \arrow[rr, "{\id^\vee}^t"]  && (L^{\vee;w_0,e}, \sigma)^t
		\end{tikzcd}
	\end{equation}
	We must show that the square on the right commutes. It follows from Lemmas \ref{lem5.10} and \ref{lem5.11}, as well as the definition of $(\cdot)^\iota$, that the square on the left commutes. It then suffices to show that the outer square commutes. We will actually prove a stronger statement, which is that the following square commutes.
	\begin{equation}\label{outersquare}
		\begin{tikzcd}
			L^{e,w_0} \arrow[rr, "\zeta"] \arrow[dd, "\varPsi_\mathbf{-i}"]  && L^{w_0,e} \arrow[dd, "\varPsi_\sigma"] \\
			\\
			L^{\vee;e,w_0} \arrow[rr, "{\zeta^\vee}"]  && L^{\vee; w_0,e}
		\end{tikzcd}
	\end{equation}
	Let $x=x_{\mathbf{-i}}(t_1,\dots, t_m)\in L^{e,w_0}$. By \cite[Eq (1.21)]{FZ}, we are given formulas for changes of variables between the factorization parameters $t_k$'s and the twisted minors $\Delta_k\circ \zeta$:
	\[
		\Delta_k\circ \zeta(x_{\mathbf{-i}}(t_1,\dots, t_m))=\Delta_{s_{i_1}\cdots s_{i_{k}}\omega_{i_k},\omega_{i_k}}\circ \zeta (x_{\mathbf{-i}}(t_1,\dots, t_m))
		=\prod_{l\geqslant k^+ }t_l^{\langle \alpha_{i_l}^\vee, s_{i_{l+1}}\cdots s_{i_k}\omega_{i_k}\rangle}.
	\]
	Note that the twist map $\zeta$ differs by a transpose $(\cdot)^T$ from the one in \cite[Eq (1.21)]{FZ}. From these we derive
	\begin{align}
		\label{firsttwist} \Delta^\vee_k\circ \varPsi_\sigma\circ \zeta(x) & =\left(\prod_{l\geqslant k^+ }t_l^{\langle \alpha_{i_l}^\vee, s_{i_{l+1}}\cdots s_{i_k}\omega_{i_k}\rangle}\right)^{d_{i_k}} \\
		\label{secondtwist} \Delta_k^\vee\circ \zeta^\vee\circ\varPsi_\mathbf{i}(x)& =\prod_{l\geqslant k^+ }t_l^{d_{i_l}\langle \alpha_{i_l}, s_{i_{l+1}}\cdots s_{i_k}\omega_{i_k}^\vee\rangle}.
	\end{align}
	But $d_{i_k}\langle \alpha_{i_l}^\vee, s_{i_l}\cdots s_{i_k}\omega_{i_k}\rangle=( \alpha_{i_l}^\vee, s_{i_{l+1}}\cdots s_{i_k}\omega_{i_k}^\vee)= d_{i_l}\langle \alpha_{i_l}, s_{i_l}\cdots s_{i_k}\omega_{i_k}^\vee\rangle$. So the expressions \eqref{firsttwist} and \eqref{secondtwist} are equal. Since this is true for all minors $\Delta_k$, we have found that \eqref{goaldiag} commutes.	\end{proof}

As a consequence of Proposition \ref{Pro:compareCluster}, Proposition \ref{dmovesok} and Theorem \ref{Main}, we immediately obtain the following. Recall that, for any double reduced words $\mathbf{i}, \mathbf{i'}$ for $(w_0,e)$, the toric charts $x_\mathbf{i},x_\mathbf{i'}, \sigma(\mathbf{i}),\sigma(\mathbf{i'})$ on $G^{w_0,e}$ are all positively equivalent.

\begin{Cor}\label{allcomparisonthesame}
	Let $\mathbf{i}, \mathbf{i'}$ be double reduced words for $(w_0,e)$, and let $\sigma\in |\sigma(\mathbf{i})|$ and $\sigma'\in |\sigma(\mathbf{i'})|$. Consider the (rational) comparison maps
	\[
	\varPsi_{\mathbf{i}}, \varPsi_{\mathbf{i'}}, \varPsi_{\sigma}, \varPsi_{\sigma'}\colon G^{w_0,e} \to G^{\vee;w_0,e}.
	\]
	For any toric charts $\theta\in [x_\mathbf{i}]$ and $\theta^\vee\in [x^\vee_{\mathbf{i}}]$ on $G^{w_0,e}$ and  $G^{\vee;w_0,e}$, respectively, the tropicalized maps
	\[
	\varPsi_{\mathbf{i}}^t,\varPsi_{\mathbf{i'}}^t, \varPsi_{\sigma}^t, \varPsi_{\sigma'}^t\colon (G^{w_0,e},\theta)^t \to (G^{\vee;w_0,e},\theta^\vee)^t
	\]
	are all equal.
\end{Cor}

\subsection{The BK Potentials and the BK Cones}

In this section we recall the Berenstein-Kazhdan (BK) potential, and the BK cone, as described in \cite{BKII}. Let $G$ be a semisimple algebraic group as before. Fix a double reduced word $\mathbf{i}$ for $(w_0,e)$ as in the previous sections.

\begin{Def}
  On the double Bruhat cell $G^{w_0,e}$, the BK potential $\Phi_{BK}$ is the following function: 
  \begin{equation}\label{BKp}
    \Phi_{BK}=\sum_{i\in \bm{I}}\frac{\Delta_{w_0\omega_i,s_i\omega_i}+\Delta_{w_0s_i\omega_i,\omega_i}}{\Delta_{w_0\omega_i,\omega_i}}\in \mathbb{Q}[G^{w_0,e}]. 
  \end{equation}
\end{Def}

\begin{Rmk} If $G$ is not simply connected, generalized minors of the form $\Delta_{u\omega_i,v\omega_i}$ are not in general functions on $G$. However, $\Phi_{BK}$ is: Suppose $\widehat{G}$ is the universal cover of $G$ with $\mathcal{p}\colon \widehat{G}\to G$ the covering map. Then the right-hand side of \eqref{BKp} is well defined on $\widehat{G}$ and is invariant under the action of any element belonging to $\ker \mathcal{p}$. Thus $\Phi_{BK}$ descends to a function on $G$.
\end{Rmk}

\begin{Rmk}
    In \cite{BKII} the authors give a more conceptual definition of $\Phi_{BK}$, using the notion of $\chi$-linear functions. Equivalence with our definition is shown by \cite[Corollary 1.25]{BKII}.
\end{Rmk}

Let $F=\sum F_i$ be the sum of the negative root vectors associated with the simple roots. Let $i^*$ be the index of the simple root $-w_0\alpha_i$, then $F_{i*}$ is the negative root vector corresponding to the root $\alpha_{i^*}:=-w_0\alpha_i$. Let $\rho=\frac{1}{2}\sum_{\alpha>0} \alpha=\sum \omega_i$ be the Weyl vector. The following proposition gives the expression of the BK potential used in \cite{ABHL}:
\begin{Pro}
\label{rewriteBK}
  The BK potential on $G$ can be rewritten as:
  \[
    \Phi_{BK}=\sum_{i\in \bm{I}}\frac{F_{i^*}\cdot\Delta_{w_0\omega_i,\omega_i}+\Delta_{w_0\omega_i,\omega_i}\cdot F_{i}}{\Delta_{w_0\omega_i,\omega_i}}=\frac{F\cdot \Delta_{w_0\rho,\rho}+\Delta_{w_0\rho,\rho}\cdot F}{\Delta_{w_0\rho,\rho}} .
  \]
\end{Pro}
\begin{proof}
  Since $\Delta_{w_0\omega_i,\omega_i}\cdot F_i=\Delta_{w_0\omega_i,s_i\omega_i}$ for the right action, and $F_{i^*}\cdot\Delta_{w_0\omega_i,\omega_i}=\Delta_{w_0s_i\omega_i,\omega_i}$ for the left action, we get the first equality. To show the second equality, one uses:
  \[
    \Delta_{w_0\rho,\rho}(g)=\Delta_{\rho}(\overline{w_0}^{-1}g)=([\overline{w_0}^{-1}g]_0)^{\rho}=\prod_{i\in\bm{I}}([\overline{w_0}^{-1}g]_0)^{\omega_i}=\prod_{i\in\bm{I}}\Delta_{w_0\omega_i,\omega_i}(g),
  \]
  as well as $\Delta_{w_0\omega_i,\omega_i}\cdot F_j=0\ (j\neq i)$ for the right action, and $F_{j^*}\cdot\Delta_{w_0\omega_i,\omega_i}=0\ (j^* \neq i^*)$ for the left action.
\end{proof}

\begin{Pro}
  For $(h,z)\in H\times L^{w_0,e}$, the BK potential has the form:
  \begin{equation}\label{BKonL}
    \Phi_{BK}(hz)=\sum_i \left( \Delta_{w_0\omega_i,s_i\omega_i}(z)+h^{-w_0\alpha_i}\Delta_{w_0s_i\omega_i,\omega_i}(z)\right).
  \end{equation}
\end{Pro}
\begin{proof}
  We only need to consider the case when $G$ is simply connected. By \eqref{BonusIdentity}, we have
  \[
    \Delta_{w_0\omega_i,s_i\omega_i}(hz)=h^{w_0\omega_i}\Delta_{w_0\omega_i,s_i\omega_i}(z), \quad \Delta_{w_0s_i\omega_i,\omega_i}(hz)=h^{w_0s_i\omega_i}\Delta_{w_0s_i\omega_i,\omega_i}(z),
  \]
  \[
    \Delta_{w_0\omega_i,\omega_i}(hz)=h^{w_0\omega_i}\Delta_{w_0\omega_i,\omega_i}(z).
  \]
  Since $h^{s_i\omega_i-\omega_i}=h^{-\alpha_i}$ and $\Delta_{w_0\omega_i,\omega_i}(z)=1$ for $z\in L^{w_0,e}$, we get the desired form. 
\end{proof}

Note that the potential $\Phi_{BK}$ is a positive function on the positive variety $(G^{w_0,e},[x_{\mathbf{i}}])$. On each positive chart $\theta\in [x_{\mathbf{i}}]$, define the \emph{BK cone} of $G$ as:
\[
  \mathcal{C}_{\theta}^G:=(G^{w_0,e},\theta,\Phi_{BK})^t = \{\zeta \in (G^{w_0,e},\theta)^t \mid \Phi_{BK}^t(\zeta)\geqslant 0\}.   
\]
Let us denote
\[
	\mathcal{L}_{\mathbf{i}}:=(G^{w_0,e},x_{\mathbf{i}})^t \cong X_*(H)\times \mathbb{Z}^m
\]
where the isomorphism to the integer lattice comes from the standard coordinates on the (split) toric chart. Let $\{e_i\}_m$ be the standard basis of $\mathbb{Z}^m$ and $\xi=\sum \xi_i e_i\in \mathbb{Z}^m$. Then we abbreviate
\[
  \mathcal{C}_{\mathbf{i}}^{G}:=\mathcal{C}_{x_\mathbf{i}}^G=\{(\lambda^\vee,\xi)\in \mathcal{L}_\mathbf{i} \mid \Phi_{BK}^t(\lambda^\vee,\xi)\geqslant 0\}.   
\]
Let $\mathcal{L}_\mathbf{i}(\mathbb{R}) = \mathcal{L}_\mathbf{i}\otimes \mathbb{R}$ be the real points of $\mathcal{L}_\mathbf{i}$, and let $\mathcal{C}_{\mathbf{i}}^G(\mathbb{R})\subset \mathcal{L}_\mathbf{i}(\mathbb{R})$ be the real cone cut out by the same inequalities as $\mathcal{C}_\mathbf{i}^G$. One can show that 
\begin{equation}\label{1-connectedforall}
  \mathcal{C}_{\mathbf{i}}^G(\mathbb{R})=\mathcal{C}_{\mathbf{i}}^{\widehat{G}}(\mathbb{R}) \subset \mathfrak{h}\times \mathbb{R}^m \ \text{~and~}\ \mathcal{C}_{\mathbf{i}}^G=\mathcal{C}_{\mathbf{i}}^{\widehat{G}}(\mathbb{R})\cap \mathcal{L}_\mathbf{i}.
\end{equation}
Similarly, define $\mathcal{L}^\vee_\mathbf{i}, \mathcal{C}_{\mathbf{i}}^{G^\vee}$, etc. Maps between the integral points of tropical varieties extend to maps of their real points.

\begin{Rmk}
  Here we consider the non-strict BK cone, rather the strict one as in \cite{ABHL}.
\end{Rmk}

The map $\psi_{\mathbf{i}}\colon \mathcal{L}_\mathbf{i}\to \mathcal{L}^\vee_\mathbf{i}$ in Proposition \ref{dmovesok} allows us to compare the potential cones for $G$ and $G^\vee$. This comparison will be discussed in detail in Theorem \ref{compare}. Let us first give an example:
\begin{Ex}
  Let $G=\SL_2$ and $H$ be the subgroup of diagonal matrices as in Example \ref{Ex:SL2}.  We have the following factorization and potential:
  \[
    x=
    \begin{bmatrix} 
      a & 0 \\
      0 & a^{-1} 
    \end{bmatrix}
    \begin{bmatrix}
      t^{-1} & 0 \\
      1 & t
    \end{bmatrix}=
    \begin{bmatrix}
      at^{-1} & 0 \\
      a^{-1}  & a^{-1}t
    \end{bmatrix};\quad
    \Phi_{BK}=t+\frac{a^2}{t}.
  \]
  Recall $X_*(H)=\mathbb{Z}\alpha^\vee$. Then the $BK$ cone is cut out by the following inequalities:
  \[
    \min\{\langle e_1^*,\xi_1 e_1\rangle,\langle e_1^*,-\xi_1 e_1\rangle+\langle x\alpha^\vee,\alpha\rangle \}\geqslant 0,
  \]
  where  $e_1^*$ is the dual of $e_1$. In other words, 
  \[
    \mathcal{C}_{\mathbf{i}}^{\SL_2}=\{(x\alpha^\vee,\xi_1e_1)\in X_*(H)\times \mathbb{Z}\mid 2x\geqslant \xi_1\geqslant 0\}.
  \]
  Note $X^*(H)=\mathbb{Z}\omega$. Now, for $G^\vee=\PSL_2$, the $BK$ cone is given by the inequalities:
  \[
    \min\{\langle e_1^*,\xi_1 e_1\rangle, \langle e_1^*,-\xi_1 e_1\rangle+\langle x\omega,\alpha^\vee\rangle \}\geqslant 0.
  \]
  Therefore,
  \[
    \mathcal{C}_{\mathbf{i}}^{\PSL_2}=\{(x\omega,\xi_1e_1)\in X^*(H)\times \mathbb{Z}\mid x\geqslant \xi_1\geqslant0\}.
  \]
  The lattice cones $\mathcal{C}_{\mathbf{i}}^{\SL_2}$ and $\mathcal{C}_{\mathbf{i}}^{\PSL_2}$ are depicted in Figure \ref{fig:latticeCones}.

  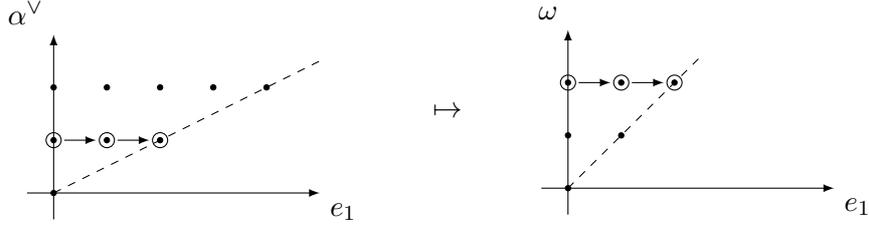
\begin{figure}[htb]
  \[
    \begin{tikzpicture}
      [dot/.style={circle,draw=black, fill,inner sep=0.7pt}, scale=0.7, baseline={([yshift=-.5ex]current bounding box.center)}]

    \foreach \x in {0,...,4}{
      \node[dot] at (\x,2){ };
      }
      \foreach \x in {0,...,2}{
      \node[dot] at (\x,1){ };
      \node[circle, draw=black, inner sep=2pt] at (\x,1){ };
      }
      \node[dot] at (0,0){};

      \draw[->,-latex] (0.2,1) -- (0.8,1);
      \draw[->,-latex] (1.2,1) -- (1.8,1);

      \draw[dashed] (0,0) -- (5,2.5);
    \draw[->,-latex] (0,-0.5) -- (0,3) node[anchor=south east] {$\alpha^\vee$};
    \draw[->,-latex] (-0.5,0) -- (5,0) node[anchor=north west] {$e_1$};
    \end{tikzpicture}\quad \quad \mapsto \quad \quad
    \begin{tikzpicture}
      [dot/.style={circle,draw=black, fill,inner sep=0.7pt}, scale=0.7, baseline={([yshift=-.5ex]current bounding box.center)}]

    \foreach \x in {0,...,2}{
      \node[dot] at (\x,2){ };
      \node[circle, draw=black, inner sep=2pt] at (\x,2){ };
      }
      \foreach \x in {0,...,1}{
      \node[dot] at (\x,1){ };
      }
      \node[dot] at (0,0){};

      \draw[->,-latex] (0.2,2) -- (0.8,2);
      \draw[->,-latex] (1.2,2) -- (1.8,2);

      \draw[dashed] (0,0) -- (2.5,2.5);
    \draw[->,-latex] (0,-0.5) -- (0,3) node[anchor=south east] {$\omega$};
    \draw[->,-latex] (-0.5,0) -- (5,0) node[anchor=north west] {$e_1$};
    \end{tikzpicture}
  \]
    \caption{Comparison of the lattice cones for $G = \SL_2$ and $G^\vee = \PSL_2$.}
    \label{fig:latticeCones}
  \end{figure}
  
  The map $\psi_{\mathbf{i}}$ is portrayed in Figure \ref{fig:latticeCones}. It
  sends the circled points in $\mathcal{C}_{\mathbf{i}}^{\SL_2}$ to the circled points in $\mathcal{C}_{\mathbf{i}}^{\PSL_2}$. It also sends arrows to arrows; these correspond to crystal operations and will be discussed in Section \ref{section;crystal}.
\end{Ex}

\begin{Rmk}
  The map $\psi_{\mathbf{i}}$ was discussed in \cite{FH,KMc}. To check that the definitions in \cite{FH} and \cite{KMc} coincide, refer to \cite[Remark 2.14, 4.3]{KMc}.
\end{Rmk}

\begin{Pro}\cite[Proposition 4.11]{BZ01}
\label{bestcharts}
	For a double reduced word $\mathbf{i}$ of $(w_0,e)$, let $z=x_{\mathbf{i}}(t_1,\dots,t_m)$ and $\omega_{i^*}=-w_0\omega_i$. Then
	\[
		\Delta_{w_0\omega_{i_m},s_{i_m}\omega_{i_m}}(z)=t_m; \quad \Delta_{w_0s_{i_1^*}\omega_{i_1^*},\omega_{i_1^*}}(z)=t_1^{-1}.
	\]
\end{Pro}
Together with Proposition \ref{Pro:compareCluster} and Theorem \ref{Main}, we can now compare $\Phi_{BK}$ and $\Phi_{BK}^\vee$:
\begin{Thm}\label{compare}
	The natural real extension of $\psi_{\mathbf{i}}$ restricts to an isomorphism of real BK cones $\mathcal{C}_{\mathbf{i}}^G(\mathbb{R})\to \mathcal{C}_{\mathbf{i}}^{G^\vee}(\mathbb{R})$. Moreover, the map $\psi_{\mathbf{i}}$ restricts to an injective map of integral BK cones: $\mathcal{C}_{\mathbf{i}}^G\to \mathcal{C}_{\mathbf{i}}^{G^\vee}$.
\end{Thm}
\begin{proof}
	We introduce the following notation for simplicity: 
	\begin{align*}
		p_i:=\Delta_{w_0\omega_i,s_i\omega_i}, \quad &  q_i:=\Delta_{w_0s_i\omega_i,\omega_i}\in \mathbb{Q}[L^{w_0,e}]; \\
		 p_i^\vee:=\Delta_{w_0\omega_i^\vee,s_i\omega_i^\vee}, \quad &  q_i^\vee:=\Delta_{w_0s_i\omega_i^\vee,\omega_i^\vee}\in \mathbb{Q}[L^{\vee;w_0,e}].
	\end{align*}
	For each index $k$, one can choose a double reduced word $\mathbf{i}_k=(i_1,\dots,i_m)$ such that $|i_m|=k$, and $\mathbf{i}_{k^*}=(i_1,\dots,i_m)$ such that $|i_1|=k^*$. Let $\alpha_{k^*}$ denote the function which sends $h$ to $h^{-w_0\alpha_k}$, and let 
	\[
	(\alpha_{k^*}\cdot q_k)(h,z) = \alpha_{k^*}(h) q_k(z).\] Then by Proposition \ref{bestcharts}, the function
  \begin{align*}
    \left((\alpha^\vee_{k^*}\cdot q_k^\vee)\circ  \varPsi_{\mathbf{i}_{k^*}}\right)^t\colon \mathcal{L}_{\mathbf{i}_{k^*}}&\to \mathbb{Z}
    \end{align*}
    can be written
    \begin{align*}
     \left(\sum a_i\omega_i^\vee,\sum \xi_ie_i\right)&\mapsto d_{k^*}\left(a_{k^*}-\xi_1\right).
  \end{align*}  
  From Proposition \ref{bestcharts} one also has:
	\[
		d_k\left(\alpha_{k^*}\cdot  q_k\right)^t\colon \mathcal{L}_{\mathbf{i}_{k^*}}\to \mathbb{Z}\ :\ \left(\sum a_i\omega_i^\vee,\sum \xi_ie_i\right)\mapsto d_{k^*}\left(a_{k^*}-\xi_1\right).
	\]
	Thus we get:
	\[
		\left((\alpha^\vee_{k^*}\cdot  q_k^\vee)\circ \varPsi_{\mathbf{i}_{k^*}}\right)^t=d_{k^*}\left(\alpha_{k^*}\cdot q_k\right)^t.
	\]
	Similarly, for the other terms, we get $(p_k^\vee\circ \varPsi_{\mathbf{i}_{k}})^t=d_kp_k^t$, where we write $p_k(h,z) = p_k(z)$.
	
	Then by Corollary \ref{allcomparisonthesame}, we have
	\begin{equation}
	\label{scalingnicely}
		(p_k^\vee\circ \varPsi_{\sigma})^t=d_kp_k^t; \quad \left((\alpha^\vee_{k^*}\cdot  q_k^\vee)\circ \varPsi_{\sigma}\right)^t=d_{k^*}\left(\alpha_{k^*}\cdot  q_k\right)^t,
	\end{equation}
	where $\sigma\in |\sigma(\mathbf{i'})|,$ for any double reduced word $\mathbf{i'}$ for $(w_0,e)$. From \eqref{scalingnicely} and \eqref{1-connectedforall}, a point $x= (h,z) \in \mathcal{L}_\mathbf{i}(\mathbb{R})$ has $\varPsi_\sigma^t(x)\in \mathcal{C}^{G^\vee}_\sigma(\mathbb{R})$ if and only if
	\[
	    d_k p_k^t(z) \geqslant 0 \ \text{~and~}\ d_{k^*}\left(\alpha_{k^*}\cdot  q_k\right)^t(h,z) \geqslant 0, \quad \forall k\in {\bm I}.
	\]
	Dividing both sides of each equation by $d_k$, this is equivalent to the condition that $x\in \mathcal{C}^G_\sigma(\mathbb{R})$. 
	
	Again by Corollary \ref{allcomparisonthesame}, we can replace $\psi_\sigma$ with $\psi_\mathbf{i}$. In particular, restricting to the integral cone $\mathcal{C}_\mathbf{i}^G$, the map
	\[
		\psi_{\mathbf{i}}\colon \mathcal{C}_{\mathbf{i}}^{G}=\mathcal{C}_{\mathbf{i}}^{G}(\mathbb{R})\cap \mathcal{L}_\mathbf{i}\to \mathcal{C}_{\mathbf{i}}^{G^\vee}=\mathcal{C}_{\mathbf{i}}^{G^\vee}(\mathbb{R})\cap \mathcal{L}^\vee_\mathbf{i}
	\]
	is an injection of cones.
\end{proof}

  We give an direct computation for the comparison of the BK cones for $\SO(5)$ and $\Sp(4)$ in Appendix \ref{B2vsC2}. We immediately have the following counterpart of Corollary \ref{allcomparisonthesame}.
  
  \begin{Cor}\label{allcomparisonthesamecone}
	Let $\mathbf{i}$ be a double reduced word for $(w_0,e)$, and consider the (rational) comparison map
	\[
	\varPsi_{\mathbf{i}}\colon G^{w_0,e} \to G^{\vee;w_0,e}.
	\]
	For any toric charts $\theta\in [x_\mathbf{i}]$ and $\theta^\vee\in [x^\vee_{\mathbf{i}}]$ on $G^{w_0,e}$ and  $G^{\vee;w_0,e}$, respectively, the tropicalized map $\varPsi_{\mathbf{i}}^t$ restricts to an injection of cones
	\[
	\varPsi_{\mathbf{i}}^t\colon \mathcal{C}_{\theta}^G \to \mathcal{C}_{\theta^\vee}^{G^\vee}.
	\]
\end{Cor}

\subsection{Crystal Structure} 
\label{section;crystal}
In this section, we will recall the crystal structure on the BK cones, and show that $\psi_{\mathbf{i}}$ respects this structure, in the sense described in Theorem \ref{crystalstructure}. We continue with notation from the previous section. In particular, let us fix a double reduced word $\mathbf{i}$ for $(w_0,e)$ of $W$ as before.

Recall from \cite{BKII} that $G^{w_0,e}$ has the structure of a positive decorated geometric crystal, which gives $\mathcal{C}_\mathbf{i}^{G}$ the structure of a normal Kashiwara crystal. This means that there are  maps
\begin{align*}
	\wt\colon& \mathcal{C}_\mathbf{i}^{G} \to P^\vee; \\
	\varepsilon_i,\varphi_i\colon& \mathcal{C}_\mathbf{i}^{G} \to \mathbb{Z}\sqcup \{ -\infty \};\\
	\tilde{e}_i,\tilde{f}_i\colon&  \mathcal{C}_\mathbf{i}^{G} \sqcup\{\oslash\} \to \mathcal{C}_\mathbf{i}^{G} \sqcup\{\oslash\},
\end{align*}
where $i\in \bm{I}$, and $\oslash$ is a ghost element. The maps satisfy certain axioms, see \cite{MK93}.

Consider the projection  
\[
  \hw\colon G^{w_0,e}=H\times L^{w_0,e}\to H.
\]
Then for simply connected $G$, Proposition \ref{charaofz} and \eqref{BonusIdentity} imply
\[
		\Delta_{w_0\omega_i,\omega_i}(hz)=h^{w_0\omega_i}\Delta_{w_0\omega_i,\omega_i}(z)=h^{w_0\omega_i},
\]
which means that the map $\hw:G^{w_0,e}\to H$ can also be characterized by the following property:
\begin{equation}\label{hwproperty}
	\left(\hw(g)\right)^{w_0\omega_i}=\Delta_{w_0\omega_i,\omega_i}(g), \quad \forall g\in G^{w_0,e}, i\in \bm{I}.  
\end{equation}

Let $\hw^{t}:\mathcal{C}^G_\mathbf{i} \to H^t$ be the restriction to $\mathcal{C}^G_\mathbf{i}$ of the tropicalization of $\hw$. For simplicity, let
\[
	\hw^{-t} (\lambda^\vee):=(\hw^t)^{-1}(\lambda^\vee)\subset \mathcal{C}^G_\mathbf{i}
\]
denote the preimage of $\lambda^\vee$ under the map $\hw^{t}$. 
\begin{Thm}\cite[Theorem 6.15]{BKII}\label{Claim6.9}	
With respect to the chart $x_{\mathbf{i}}$, the image of $\hw^t$ lies in set of the dominant weights $X_*^+(H)$. So there is a direct decomposition:
	\[
		\mathcal{C}_{\mathbf{i}}^G=\bigsqcup_{\lambda^\vee\in X_*^+(H)}\hw^{-t}(\lambda^\vee).
	\]
	Moreover,
	\[
		\hw^{-t}(\lambda^\vee)\cong B_{\lambda^\vee},
	\]
	where $B_{\lambda^\vee}$ is the crystal associated with the irreducible $G^\vee$-module with highest weight $\lambda^\vee$.
\end{Thm}

From \cite[Lemma 3.10]{TN}, the operators $\tilde{e}_i$ and $\tilde{f}_i$ on the crystal $\mathcal{C}_\mathbf{i}^{G}$ can be written explicitly as follows. 
Let $v_1,\dots, v_m$ be the standard basis of $\mathbb{Z}^m$. Let 
\[
x:=\left(h,\sum\xi_j v_j\right)\in \mathcal{C}^{G}_\mathbf{i} \subset \mathcal{L}=X_*(H)\times \mathbb{Z}^m.\]
Then the crystal operators on $\mathcal{C}_\mathbf{i}^G$ are given by
\begin{align*}
	\tilde{f}_i\left(h,\sum\xi_jv_j\right) &=
		\begin{cases}
			\Big(h,\sum\xi_jv_j+v_{n_f}\Big) & \text{if~}  \Big(h,\sum\xi_jv_j+v_{n_f}\Big)\in \mathcal{C}_\mathbf{i}^G,\\
			\oslash &\text{else;}
		\end{cases} \\
	\tilde{e}_i\left(h,\sum\xi_jv_j\right) &=
		\begin{cases}
			\Big(h,\sum\xi_jv_j-v_{n_e}\Big) & \text{if~}  \Big(h,\sum\xi_jv_j-v_{n_e}\Big)\in \mathcal{C}_\mathbf{i}^G,\\
			\oslash &\text{else.}
		\end{cases}
\end{align*}
The indices $n_f=n_f(x,i)$ and $n_e=n_e(x,i)$ are given by:
\begin{align}
	n_f & :=\min\left\{ l \ \Big|\  1\leqslant l \leqslant m,~i_l=i, X_l=\min_{l'}\{X_{l'}\mid i_{l'}=i\} \right\}; \label{eq;actionE} \\
	n_e & :=\max\left\{ l \ \Big|\  1\leqslant l \leqslant m,~i_l=i, X_l=\min_{l'}\{X_{l'}\mid i_{l'}=i\} \right\}, \nonumber
\end{align}
where, for an index $l$, 
\[
	X_l(x,i)=\sum_{k=1}^{l} a_{i_k,i} \xi_{k}.
\]
As in Remark \ref{aboutsign}, here we drop the minus signs from the $i_j$'s.
Observe that, if $x,\tilde{e}_i x\in \mathcal{C}^G_\mathbf{i}$, then
\begin{equation}
\label{operatorapplication}
n_e(x,i)=n_e (\tilde{e}_i x,i).
\end{equation}

\begin{Thm}
\label{crystalstructure}
	Consider the map $\psi_\mathbf{i}\colon \mathcal{C}^G_\mathbf{i} \to \mathcal{C}^{G^\vee}_\mathbf{i}$ as in Theorem \ref{compare}. Then for any $i\in {\bm I}$, 
	\[
		\psi_\mathbf{i}\circ\tilde{e}_i=\tilde{e}_i^{d_i}\circ\psi_\mathbf{i},\quad \psi_\mathbf{i}\circ \tilde{f}_i=\tilde{f}_i^{d_i}\circ \psi_\mathbf{i},
	\]
	where we write $\tilde{e}_i,\tilde{f}_i$ for the crystal operators in both $\mathcal{C}^G_\mathbf{i}$ and $\mathcal{C}^{G^\vee}_\mathbf{i}$.
\end{Thm}

\begin{proof}
    We will prove the statement for $\tilde{e}_i$; the one for $\tilde{f}_i$ follows immediately from the crystal axioms. Assume $x, \tilde{e}_i x\in \mathcal{C}^G_\mathbf{i}$. By Theorem \ref{compare}, we have $\psi_\mathbf{i}(x), \psi_\mathbf{i}(\tilde{e}_i x) \in \mathcal{C}^{G^\vee}_\mathbf{i}$.   
    
    From $i_{n_e(x,i)}=i$, one sees immediately that
    \[
    \psi_\mathbf{i}\left(\tilde{e}_i\left(h,\sum\xi_j v_j\right) \right) = \left( \psi(h), \sum d_{i_j} \xi_j v_j- d_i v_{n_e(x,i)} \right).
    \]
    By convexity of $\mathcal{C}^{G^\vee}_\mathbf{i}$, the lattice points between $\psi_\mathbf{i}(x)$ and $ \psi_\mathbf{i}(\tilde{e}_i x)$ are contained in $\mathcal{C}^{G^\vee}_\mathbf{i}$ as well. We will show that these are exactly the points obtained by repeatedly applying the operator $\tilde{e}_i$ in $\mathcal{C}^{G^\vee}_\mathbf{i}$.
    
    First, by the description above
    \[
    \tilde{e}_i\left( \psi_\mathbf{i}\left(h,\sum\xi_j v_j\right) \right) = \left( \psi(h), \sum d_{i_j} \xi_j v_j - v_{n_e(\psi_\mathbf{i}(x),i)} \right)\in \mathcal{C}^{G^\vee}_\mathbf{i}.
    \]
    Assume that $n_e(x,i)=n_e(\psi_\mathbf{i}(x),i)$. From \eqref{operatorapplication} applied to the crystal $\mathcal{C}^{G^\vee}_\mathbf{i}$, one gets that 
    \[
    n_e(\tilde{e}_i^k \psi_\mathbf{i} (x),i) = n_e( \psi_\mathbf{i} (x),i)= n_e(x,i),
    \]
    for $0\leqslant k < d_i$. So, applying $\tilde{e}_i$ repeatedly gives
    \[
    \psi_\mathbf{i} (\tilde{e}_i x) = \tilde{e}_i^{d_i} \psi_\mathbf{i}(x).
    \]
    
    It remains to show that $n_e(x,i)=n_e(\psi_\mathbf{i}(x),i)$. Indeed,
    \begin{align*}
       X_l(\psi_\mathbf{i}(x),i) & = \sum_{k=1}^l (a^T)_{i_k,i} d_{i_k} \xi_k 
        = \sum_{k=1}^l a_{i,i_k} d_{i_k}\xi_k 
         = \sum_{k=1}^l d_i a_{i_k,i} \xi_k
         = d_i X_l(x,i).
    \end{align*}
    So,
    \begin{align*}
        n_e(x,i) & = \max\left\{ l \ \Big|\  1\leqslant l \leqslant m,~i_l=i, X_l(x,i) =\min_{l'}\{X_{l'}(x,i)\mid i_{l'}=i\} \right\} \\
        & = \max\left\{ l \ \Big|\  1\leqslant l \leqslant m,~i_l=i, d_i X_l(x,i)=\min_{l'}\{d_i X_{l'}(x,i) \mid i_{l'}=i\} \right\} \\
        & = \max\left\{ l \ \Big|\  1\leqslant l \leqslant m,~i_l=i, X_l(\psi_\mathbf{i}(x),i )=\min_{l'}\{X_{l'} (\psi_\mathbf{i}(x),i)\mid i_{l'}=i\} \right\} \\
        & = n_e(\psi_\mathbf{i}(x),i).
    \end{align*}
    This proves the claim.
\end{proof}
\begin{Rmk}
	Restricting to $\hw^{-t}(\lambda^\vee)$ and identifying $\hw^{-t}(\lambda^\vee)$ with $B_{\lambda^\vee}$,  Theorem \ref{crystalstructure} is a special case of Kashiwara's theorem as in \cite[Theorem 5.1]{MK96}. Note that Theorem 2.6 in \cite{FH} is also a  special case of Kashiwara's theorem, as indicated by the authors. 
\end{Rmk}

\section{Poisson-Lie Duality and Langlands Duality} \label{poissonsection}

In this section, we pass to the complex points of our varieties, and assume $G$ is a (not necessarily simply connected) semisimple complex Lie group. Let $K$ be the compact real form of $G$. In \cite{ABHL}, a Poisson manifold $PT(K^*)$ was constructed as a `tropical-limit' of the dual Poisson-Lie group $K^*$. We use our results here to compute the symplectic volume of symplectic leaves of $PT(K^*)$.

\subsection{The Poisson Manifold \texorpdfstring{$PT(K^*)$}{PT(K*)}}

In this section we define the Poisson manifold $PT(K^*)$. Let us first fix some notation.

 It is shown in \cite{LW} that there is a standard Poisson-Lie structure $\pi_K$ on $K$ which depends only on a choice of invariant inner product $(\cdot,\cdot)$ on $\mathfrak{g}$. Let $(K^*,\pi_{K^*})$ be the Poisson-Lie group dual to $(K,\pi_K)$. 

The Lie group $K^*$ can be identified with $N_-A$, where $G=N_-AK$ is the Iwasawa decomposition of $G$. We regard $N_-A$ as a subset of $B_-\subset \widehat{G}$, where $\widehat{G}$ is the universal cover of $G$. Let $\mathbf{i}=(i_1,\dots, i_{m})$ be a double reduced word of $(w_0,e)$. Recall that for the seed $\sigma(\mathbf{i}),$ we have the following cluster variables on $\widehat{G}^{w_0,e}\subset B_-$ by \eqref{fullcluster}:
\[
	\Delta_{k}:=\Delta_{u_k\omega_{i_k},\omega_{i_k},} \text{~for~} k\in [1, m];\quad \Delta_k:=\Delta_{\omega_k, \omega_k} \text{~for~} k\in [-r, -1]
\]
where $u_k=s_{i_1}\cdots s_{i_{k}}$. The $\Delta_k$'s determine a toric chart 
\begin{equation}\label{clusterfor1connected}
  \sigma(\mathbf{i}) \colon(\mathbb{C}^\times)^{m+r}\to \widehat{G}^{w_0,e},
\end{equation}
which is positively equivalent to $x_{\mathbf{i}}$. 

The collection of functions
\begin{align*}
  \Delta_k, \overline{\Delta}_{k}, \text{~for~} k\in [1, m] \text{~and~} \Delta_k, \text{~for~} k\in [-r, -1]
\end{align*}
determines a real coordinate system on an open dense subset of $K^*$. Note that for $k\in [-r,-1]$, the principal minor $\Delta_k|_{K^*}$ is equal to its complex conjugate $\overline{\Delta}_k|_{K^*}$.

In \cite{ABHL}, we associated to $K^*$ a Poisson manifold $PT(K^*)$ with a constant Poisson bracket $\pi_{PT}$, which we will now describe. Define
\[
	PT(K^*,\sigma(\mathbf{i})) := \mathcal{C}^G_{\sigma(\mathbf{i})}(\mathbb{R})\times (S^1)^m,
\]
with coordinates
\begin{align*}
	(\lambda_{-r},\dots,\lambda_{-1},\lambda_{1},\dots,\lambda_{m},e^{\sqrt{-1}\varphi_{1}},\dots,e^{\sqrt{-1}\varphi_{m}})\in \mathcal{C}^G_{\sigma(\mathbf{i})}(\mathbb{R})\times (S^1)^m.
\end{align*}
Let
$PT^\circ(K^*,\sigma(\mathbf{i})):= \left(\mathcal{C}^G_{\sigma(\mathbf{i})}(\mathbb{R})\right)^\circ\times (S^1)^m$ be the Cartesian product of the interior of  $\mathcal{C}^G_{\sigma(\mathbf{i})}(\mathbb{R})$ and $(S^1)^m$.

We introduce the \emph{detropicalization map} $\mathfrak{L}_s$: For $s<0$,
\begin{align*}
	\mathfrak{L}_s\colon \mathbb{R}^{r+m} &\to K^*\\
	(\lambda_{-r},\dots,\lambda_{-1},\lambda_{1},\dots,\lambda_{m},\varphi_{1},\dots,\varphi_{m})&\mapsto 
	\sigma(\mathbf{i})\left( e^{s\lambda_{-r}},\dots, e^{s\lambda_{-1}},e^{s\lambda_{1}-\sqrt{-1}\varphi_{1}},\dots, e^{s\lambda_{m}-\sqrt{-1}\varphi_{m}}\right),
\end{align*}
where we recall $\sigma(\mathbf{i})\colon (\mathbb{C}^\times)^{m+r}\to B_-$ is the toric chart given by the seed $\sigma(\mathbf{i})$. We now define $\pi_{PT}$ as follows.
\begin{Thm}
    The following limit exists 
    \[
	    \lim_{s\to -\infty}\left(\mathfrak{L}_s^*(s\pi_{K^*})\Big|_{PT^\circ(K^*,\sigma(\mathbf{i}))}\right).
    \]
    It extends to a unique constant Poisson bivector $\pi_{PT}$ on $PT(K^*, \sigma(\mathbf{i}))$.
     Moreover, $\pi_{PT}$ is given explicitly by
  \begin{align}
    \{\lambda_k,\varphi_p\}&=0, & &\text{for~} k\geqslant p; \label{bracket1}\\
    \{\lambda_k,\varphi_p\}&=(\omega_{i_k},\omega_{i_p})-(u_k\omega_{i_k},u_p\omega_{i_p}),& &\text{for~} k<p; \label{bracket2} \\
    \{\lambda_k,\lambda_p\}&=\{\varphi_k,\varphi_p\}=0. & &\text{for all~}k,p.\nonumber
  \end{align}
\end{Thm}
\begin{proof} The existence of the limit as a constant Poisson bivector is Theorem 6.18 of \cite{ABHL}. Let $G^*=(K^*)^{\mathbb{C}}$ be the dual Poisson-Lie group of $G$. Note that $G^*\subset B\times B_-$. Let $(\cdot)_i$ denote the pull-back of the natural projection of $G^*$ to the $i$-th factor. By \cite[Equation (14)]{ABHL}, the log-canonical part $\{\cdot,\cdot\}_{\log}$ of $\pi_{K^*}$ is
  \[
    \frac{\{(\Delta_k)_1, (\Delta_p\circ \tau)_2\}_{\log}}{(\Delta_k)_1(\Delta_p\circ \tau)_2}=(\omega_{i_k},\omega_{i_p})-(u_k\omega_{i_k},u_p\omega_{i_p}).
  \]
  By \cite[Theorem 2.6, Remark 2.8]{KZ}, we have:
  \[
    \frac{\{(\Delta_k)_1, (\Delta_p)_1\}_{\log}}{(\Delta_k)_1(\Delta_p)_1}=(\omega_{i_k},\omega_{i_p})-(u_k\omega_{i_k},u_p\omega_{i_p}),\quad \text{where~}  k\leqslant p.
  \]
  Take this into \cite[Theorem 6.16]{ABHL}, one get desired formulas.
\end{proof}
\begin{Rmk}
  Note that our normalization of the Poisson bracket here differs from that of \cite{ABHL} by a factor of $2$. We take the limit $s\to -\infty$ (rather than $s\to \infty$) because we use `$\min$' for the definition of tropicalization, rather than `$\max$' as was our convention in \cite{ABHL}.
\end{Rmk}

\subsection{Symplectic Leaves of \texorpdfstring{$PT(K^*)$}{PT(K*)}}

In this section, we study the constant Poisson bracket $\pi_{PT}$ and show that the symplectic leaves of $PT(K^*,\sigma(\mathbf{i}))$ are exactly the fibers of the highest weight map. We continue with the notation from the previous section, and fix a double reduced word $\mathbf{i}$ for $(w_0,e)$ as before. Recall ${\bm e}(\mathbf{i})\subset [1,m]$ is the set of $\mathbf{i}$-exchangeable indices. Note that $\Delta_k=\Delta_{w_0\omega_{i_k},\omega_{i_k}}$ for $k\in [1,m]\backslash {\bm e}(\mathbf{i})$. 

\label{pipt}
\begin{Pro}
	For $k\in [1,m]\backslash {\bm e}(\mathbf{i})$, the functions $\lambda_{k}$ are Casimirs of the constant bracket $\pi_{PT}$.
\end{Pro}
\begin{proof}
	Let $k\in [1,m]\backslash {\bm e}(\mathbf{i})$. By \eqref{bracket1}, we have $\{\lambda_{k},\varphi_p\}=0$ for $p\leqslant k$. All we need to show is $\{\lambda_{k},\varphi_{p}\}=0$ for $p>k$. Since $k\in [1,m]\backslash {\bm e}(\mathbf{i})$, we know $p\in [1,m]\backslash {\bm e}(\mathbf{i})$ as well. Therefore, by \eqref{bracket2},
	\[
		\{\lambda_{k},\varphi_{p}\}=(\omega_{i_k},\omega_{i_p})-(w_0\omega_{i_k},w_0\omega_{i_p})=0,
	\]
	because the bilinear form is $W$-invariant.
\end{proof}

\begin{Thm}\label{sym}
	Let $B$ be the matrix of Poisson brackets
	\[
	B=[\{\lambda_k,\varphi_{s^+}\}]_m,
	\]
	where $k,s\in [-1,-r]\cup {\bm e}(\mathbf{i})$. Then $B$ is of the form $B=DB'$, where
	\[
		D=\diag\left((\alpha_{i_1},\omega_{i_1}),\dots,(\alpha_{i_m},\omega_{i_m})\right)=(1/d_{i_1}, \dots, 1/d_{i_m}),
	\]
	and $B'$ is a upper triangular positive integer matrix with diagonal elements being $1$.
\end{Thm}
\begin{proof}
	Recall that $s_i$ is the simple reflection generated by $\alpha_i$. Note that $s_i\omega_i=\omega_i-\alpha_i$, and $s_j\omega_i=\omega_i$ for $j\neq i$. By \eqref{k+}, we have $\omega_{i_k}=\omega_{i_{k^+}}$. Thus by \eqref{bracket2}:
	\begin{align*}
		\{\lambda_k,\varphi_{k^+}\}&=(\omega_{i_k},\omega_{i_{k^+}})-(\omega_{i_k},s_{i_{k+1}}\cdots s_{i_{k^+}}\omega_{i_{k^+}})\\
		&=(\omega_{i_k},\omega_{i_k})-(\omega_{i_k},s_{i_k}\omega_{i_k})=(\alpha_{i_k},\omega_{i_k})>0.
	\end{align*}
	For $k<s^+$ and $s\neq k$, we have 
	\[
		\{\lambda_k,\varphi_{s^+}\}=\left(\omega_{i_k},\omega_{i_{s^+}}-u_k^{-1}u_{k^+}\omega_{i_{s^+}}\right)=c(\alpha_{i_k},\omega_{i_k})
	\]
	for $c\in \mathbb{Z}_{\geqslant 0}$, because $\omega_i-v\omega_i$ is a non-negative integer linear combination of $\alpha_j$'s and $(\alpha_i,\omega_j)=0$ for $i\neq j$. The bracket $\{\lambda_k,\varphi_{s^+}\}$ vanishes if $k \geqslant s^+$ by $\eqref{bracket1}$. 
\end{proof}

In what follows, we write $\hw^t_{\mathbb{R}}\colon \mathcal{C}_{\sigma(\mathbf{i})}^G(\mathbb{R})\to H^t\otimes \mathbb{R}$ for the natural extension of the tropical map $\hw^t\colon\mathcal{C}_{\sigma(\mathbf{i})}^G\to H^t$ to the real cone.

\begin{Thm}\label{symplectic leaves}
  The symplectic leaves in $PT(K^*,\sigma(\mathbf{i}))=\mathcal{C}_{\sigma(\mathbf{i})}^G(\mathbb{R}) \times (S^1)^m$ are of the form
  \[
    \hw^{-t}_{\mathbb{R}}(\lambda^\vee)\times (S^1)^m,
  \]
  for $\lambda^\vee$ in the positive Weyl chamber of $\mathfrak{g}^\vee$.
\end{Thm}

\begin{proof}
  Because of \eqref{1-connectedforall}, we only need to show the statement for simply connected $G$. Recall from \eqref{hwproperty} that
  \[
    \left(\hw(g)\right)^{w_0\omega_i}=\Delta_{w_0\omega_i,\omega_i}(g), \quad \forall g\in G^{w_0,e}, i\in \bm{I}.
  \]
	The tropicalization of $\hw$ with respect to the chart $\sigma(\mathbf{i})$ can then be written as a linear combination of the cocharacters $w_0\alpha^\vee_i\in X_*(H)$, with the tropical functions $\Delta_{w_0\omega_i,\omega_i}$ as coefficients:
	\[
		\hw^t_{\mathbb{R}}= \sum_{i} \lambda_{-i} \cdot (w_0 \alpha^\vee_i) = \sum_{i} \Delta_{w_0\omega_i,\omega_i} ^t\cdot  (w_0 \alpha^\vee_i).
	\]
	Together with the nondegeneracy of the matrix $B$ in Theorem \ref{sym}, this proves the claim.
\end{proof}

\begin{Rmk}
  There is a natural bijection between generic symplectic leaves of the $PT(K^*)$ and generic symplectic leaves of $K^*$: Let $(\cdot)^*$ be the anti-involution which picks out the compact real form $K$. Note that the symplectic leaves on $K^*$ are the level sets of the following Casimir functions \cite{LW}:
  \[
  	C^2_i(b):=\Tr\left(\rho_i\left(bb^*\right)\right), \text{~for~} b\in \im(K^*\hookrightarrow B_-\subset G) \text{~and~} i\in \bm{I},
  \]
  where $\rho_i\colon G\to V_{\omega_i}$ is the fundamental $G$-representation with highest weight $\omega_i$.

  Pick a unitary weight basis for  $V_{\omega_i}$ and let $\rho_i(b)_{j,k}$ be the $(j,k)$-matrix entry of the matrix $\rho_i(b)$ in terms of this basis. The generalized minors $\Delta_{w_0\omega_i,\omega_i}$ satisfy, for $h,h'\in H$ and $E,E'\in \mathfrak{n}$,
  \[
  \Delta_{w_0\omega_i,\omega_i}(hxh') = h^{w_0\omega_i} {h'}^{\omega_i}\Delta_{w_0\omega_i,\omega_i}(x), \qquad E\cdot \Delta_{w_0\omega_i,\omega_i} \cdot E' = 0.
  \]
  Then, computing in terms of the matrix entries of $\rho_i(b)$,
  \begin{align*}
  	C_i^2(b)&=\sum_{j,k}|(\rho_i(b)_{j,k}|^2 =\left|\Delta_{w_0\omega_i,\omega_i}(b)\right|^2+\sum_{\mathbf{j}, \mathbf{k}} \Big|c_{\mathbf{i},\mathbf{j}}\left(F_{\mathbf{j}}\cdot\Delta_{w_0\omega_i,\omega_i}\cdot F_{\mathbf{k}} \right) (b) \Big|^2\\
  	&=\left|\Delta_{w_0\omega_i,\omega_i}(b)\right|^2\left(1+\sum_{\mathbf{j}, \mathbf{k}} \left|\frac{c_{\mathbf{i},\mathbf{j}}\left(F_{\mathbf{j}}\cdot\Delta_{w_0\omega_i,\omega_i}\cdot F_{\mathbf{k}}\right)(b)}{\Delta_{w_0\omega_i,\omega_i}(b)}\right|^2\right).
  \end{align*}
   The second sum is over some non-zero sequences of indices $\mathbf{j}=(j_1,\dots,j_p)$ and $\mathbf{k}=(k_1,\dots,k_q)$. Here $F_{\mathbf{j}}$ is shorthand for $F_{j_i}\cdots F_{j_p}\in U(\mathfrak{g})$, and $c_{\mathbf{i},\mathbf{j}}\in \mathbb{C}$ are constants. 

  Recall the detropicalization map $\mathfrak{L}_s$. By \cite[Theorem 4.13]{ABHL}, under this change of coordinates, 
  \begin{equation}
  \label{tropofleaves}
  	\lim_{s\to -\infty} \frac{1}{s} \log C_i=\Delta^t_{w_0\omega_i,\omega_i} \circ \pr_1= \lambda_{-i}.
  \end{equation}
  As above, taking the limit $s\to -\infty$ (rather than $t\to \infty$) is because we use `$\min$' for the definition of tropicalization, rather than `$\max$' as was our convention in \cite{ABHL}. 
  
  This can be summarized (non-precisely) as follows: \emph{Tropicalization of symplectic leaves of $K^*$ are symplectic leaves of $PT(K^*)$.} In \cite{AHLL2}, three of the authors and Jeremy Lane build symplectic embeddings from any
  symplectic leaf of $PT(K^*)$ into the corresponding coadjoint orbit. These embeddings are not actually globally defined, but they are defined off a subset of arbitrarily small symplectic volume. We hope to establish (globally defined) symplectic embeddings in future work.
\end{Rmk}

\subsection{Comparison of Lattices}

Fix $G$ and $K$ as before, as well as a double reduced word $\mathbf{i}$ for $(w_0,e)$. Let $(K^*, \pi_{K^*})$ be the dual Poisson-Lie group to $K$. If $V$ is a $n$-dimensional real vector space and $D\subset V$, then a \emph{lattice} in $D$ is subset of $D$ of the form $L\cap D$, where $L\cong \mathbb{Z}^n$ is a lattice in $V$.

In this section, we use our results to compare two lattices on $\mathcal{C}_{\sigma(\mathbf{i})}^{G}(\mathbb{R})$. The first lattice comes from the crystal structure on $\mathcal{C}_{\sigma^\vee (\mathbf{i})}^{G^\vee}$. The integrable system on $PT(K^*,\sigma(\mathbf{i}))$ gives us the second lattice $\Lambda$, which will be called {\em Bohr-Sommerfeld} lattice. The lattice $\Lambda$ is built out of the lattice $\psi^{-1}(X^*(H))\subset \mathfrak{h}$ and a lattice in each integral symplectic leaf of $PT(K^*)$.

\begin{Rmk}
The dual Poisson-Lie group $K^*$ is the same for both $K$ and its universal cover $\widehat{K}$. But the definition of both lattices on $PT(K^*)$ involve $X^*(H)\subset\mathfrak{h}^\vee$. Therefore we are thinking of $K^*$ not just as a Poisson-Lie group, but as a Poisson Lie group which is dual to a specific integration of $\mathfrak{k}$. Then $K^*$ carries an action of $K$, which is called the dressing action.
\end{Rmk}

Now we will describe the lattice in the symplectic leaves of $PT(K^*)$. The Bohr-Sommerfeld quantization defines a lattice (integral affine structure) in the tangent spaces to leaves as follows. Assume $\lambda^\vee\in \mathfrak{h}$ is a regular dominant weight of $G$ such that 
\[
  \lambda:=  \psi(\lambda^\vee)\in X^*_+(H)\subset \mathfrak{h}^\vee.
\]
Let $\hw^t\colon \mathcal{C}^G_{\sigma(\mathbf{i})}\to X_*(H)$ be the tropicalization of $\hw\colon G^{w_0,e}\to H$ relative to the chart $\sigma(\mathbf{i})$, and let $\hw^t_{\mathbb{R}}$ be the natural extension of $\hw^t$ to the real cone $\mathcal{C}^G_{\sigma(\mathbf{i})}(\mathbb{R})$.

We write 
\[
  \mathcal{O}(\lambda^\vee):= \hw^{-t}_{\mathbb{R}}(\lambda^\vee)\times (S^1)^m\subset PT(K^*,\sigma(\mathbf{i}))
\]
for the symplectic leaf over $\lambda^\vee$, and let
\[
  \omega_{\lambda^\vee}:=(\pi_{PT}\big|_{\mathcal{O}(\lambda^\vee)})^{-1}
\]
be the symplectic form on $\mathcal{O}(\lambda^\vee)$. Let $\xi\in \hw^{-t}(\lambda^\vee)$ be the unique point in $\mathcal{C}_{\sigma(\mathbf{i})}^{G}(\mathbb{R})$ such that
\[
  \wt^t_\mathbb{R} (\xi)=\hw^t_\mathbb{R} (\xi)=\lambda^\vee.
\]
Consider the lattice $X_*(S^1)^m\subset T_1(S^1)^m$ of cocharacters of $(S^1)^m$; this lattice is generated by 
\[
  \left\{ 2\pi \frac{d}{d\varphi_{k^+}}\ \Big| \ k = 1,\dots, m \right\};
\]
recall that the angles of $PT(K^*)$ are labeled by $k^+$, where $k\in [-r,-1]\cup {\bm e}(\mathbf{i})$. Thus the following
\begin{equation}
    \label{hereisaset}
  \left\{v\in T_\xi \hw^{-t}_{\mathbb{R}}(\lambda^\vee)\mid \omega_{\lambda^\vee}(v,X_*(S^1)^m)\subset 2\pi \mathbb{Z} \right\}
\end{equation}
is a lattice in $T_\xi \hw^{-t}_{\mathbb{R}} (\lambda^\vee)$. The natural identification of $\hw^{-t}_{\mathbb{R}}(\lambda^\vee)$ with a subset of $T_\xi \hw^{-t}_{\mathbb{R}} (\lambda^\vee)$ determines the lattice $\widetilde{\Lambda}$ on $\hw^{-t}_{\mathbb{R}}(\lambda^\vee)$. Alternatively, we can think of the points of the set \eqref{hereisaset} as elements of a (scaled) dual basis to $X_*(S^1)^m$, under the pairing given by the symplectic form. In our choice of coordinates, the symplectic form is described by the matrix $B$ in Theorem \ref{sym}. So another description of the lattice $\widetilde{\Lambda}$ is
\begin{equation}\label{lambda1}
  \widetilde{\Lambda}= \left(\xi+ B \left(\mathbb{Z},\dots,\mathbb{Z} \right)^T\right)\cap \mathcal{C}_{\sigma(\mathbf{i})}^G(\mathbb{R}).
\end{equation}

Together, the lattice in $\mathfrak{h} \cong \mathfrak{h}^*$ and the lattices $\widetilde{\Lambda}$ on the integral symplectic leaves determine the \emph{Bohr-Sommerfeld lattice}
\begin{equation}
    \label{bohrlattice1}
    \Lambda := B(\mathbb{Z},\dots,\mathbb{Z})^T + \left\{\xi\in \mathcal{C}_{\sigma(\mathbf{i})}^{G}(\mathbb{R})\mid \psi\circ \hw^t_\mathbb{R} (\xi) = \psi\circ \wt^t_\mathbb{R}(\xi) =\lambda\in X^*_+(H) \right\}.
\end{equation}
In Appendix \ref{appendix2} the Bohr-Sommerfeld lattices are treated more generally. As a consequence of Lemma \ref{bohrlatticelemma} and Theorem 6.23 of \cite{ABHL}, the Bohr-Sommerfeld lattice on $PT(K^*)$ is independent of the choice of toric chart $\sigma(\mathbf{i})$.

Let us denote the real extension of the comparison map $\psi_{\sigma(\mathbf{i})}$ as
  \[
      \psi_{\sigma(\mathbf{i})}^{\mathbb{R}}\colon \mathcal{C}_{\sigma(\mathbf{i})}^G(\mathbb{R}) \to \mathcal{C}_{\sigma^\vee(\mathbf{i})}^{G^\vee}(\mathbb{R}).
  \]
\begin{Thm}\thlabel{comparisonmapF}
  The comparison map $\psi_{\sigma(\mathbf{i})}$ sends the Bohr-Sommerfeld lattice to the lattice coming from the crystal structure on $\mathcal{C}_{\sigma^\vee (\mathbf{i})}^{G^\vee}$, {\em i.e.} $\psi_{\sigma(\mathbf{i})}^{\mathbb{R}}(\Lambda)=\mathcal{C}_{\sigma^\vee(\mathbf{i})}^{G^\vee}$.
\end{Thm}

\begin{proof}
    From \eqref{lambda1} and Theorem \ref{sym}, we have
    \[
      \widetilde{\Lambda}= \left(\xi+ \left( 0,\dots,0,\frac{1}{ d_{i_1}}\mathbb{Z} , \frac{1}{ d_{i_2}}\mathbb{Z} ,\dots, \frac{1}{ d_{i_m}} \mathbb{Z} \right) \right)\cap \mathcal{C}^G_{\sigma(\mathbf{i})}(\mathbb{R}).
    \]
    From Theorem \ref{compare}, it is then clear that $\psi_{\sigma(\mathbf{i})}^{\mathbb{R}}(\widetilde{\Lambda})= (\hw^\vee)^{-t}(\lambda)$.
\end{proof}

\subsection{Comparison of Volumes}

As an application of the results of the previous section, we compare the symplectic volume of symplectic leaves of $PT(K^*)$ with that of symplectic leaves of $\mathfrak{k}^*$. We continue with the notation of the previous section.

\begin{Def}[Notation] \label{volumenotation}
  Let $L$ be a lattice in $\mathbb{R}^n$. Then $L$ induces a natural translation-invariant measure $\mu_L$ on $\mathbb{R}^n$. For a compact domain $U\subset \mathbb{R}^n$, let $\vol_L(U)$ be the volume of $U$ with respect to $L$. For a symplectic form $\omega$ on $U$, let $\vol_\omega(U)$ be the volume of $U$ with respect to the Liouville form.
\end{Def}
Recall that $\mathcal{O}(\lambda^\vee)= \hw^{-t}_{\mathbb{R}}(\lambda^\vee)\times (S^1)^m$, then
\begin{Pro}
   Using the notation just defined, we have
  \begin{equation}\label{comparevols}
    \frac{1}{(2\pi)^{m}}\vol_{\omega_{\lambda^{\vee}}}(\mathcal{O}(\lambda^\vee))=\vol_{\Lambda}(\hw^{-t}_{\mathbb{R}}(\lambda^\vee)).
  \end{equation} 
\end{Pro}
\begin{proof}
    The Liouville measure of $\omega_{\lambda}$ is a product of the translation-invariant measure $\mu_{\Lambda}$ on $\hw^{-t}_{\mathbb{R}}(\lambda^\vee)$ and $(2\pi)^m$ times the normalized Haar measure on $(S^1)^m$. The proposition follows immediately by Fubini theorem.
\end{proof}
As a consequence of \thref{comparisonmapF}, we have
\begin{equation}\label{comparevol2}
  \vol_\Lambda\left( \hw^{-t}_{\mathbb{R}}\left(\psi^{-1} (\lambda)\right)\right)= \vol_{{\hw^\vee}^{-t}(\lambda)} \left((\hw^\vee)^{-t}_{\mathbb{R}} (\lambda)\right).
\end{equation}

Recall the standard Kirillov-Kostant-Souriau Poisson structure $\pi_{\mathfrak{k}^*}$ on $\mathfrak{k}^*$. For a fixed symplectic leaf, denote by $\omega_{\mathfrak{k}^*}$ the corresponding symplectic form.

\begin{Thm}
\label{theoremcomparevolume}
  Let $\psi(\lambda^\vee)=\lambda\in X^*_+(H)$ be a regular dominant integral weight of $G$. The symplectic volume of the symplectic leaf $\mathcal{O}\left(\lambda^\vee \right)\subset PT(K^*,\sigma(\mathbf{i}))$ is equal to the symplectic volume of $\mathcal{O}_{\mathfrak{k}^*}(\sqrt{-1} \lambda)\subset \mathfrak{k}^*$, the leaf through $\sqrt{-1}\lambda\in \mathfrak{t}^*$. That is,
  \[
  \vol_{\omega_{\lambda^\vee}}\left(\mathcal{O}\left( \lambda^\vee \right)\right)= \vol_{\omega_{\mathfrak{k}^*}} (\mathcal{O}_{\mathfrak{k}^*}(\sqrt{-1} \lambda)).
  \]
\end{Thm}

\begin{proof}
Let $V_{\lambda}$ be the irreducible $G$-module with highest weight $\lambda$. Recall from Theorem \ref{Claim6.9} that $\dim(V_{\lambda})=\# {\hw^\vee}^{-t}(\lambda)$, the number of lattice points in ${\hw^\vee}^{-t}(\lambda)$. Recall also that Weyl dimension formula is:
\[
  \dim(V_{\lambda})= \prod_{\alpha>0}\frac{(\lambda+\rho,\alpha)}{(\rho,\alpha)},
\]
where $\rho$ is the half-sum of positive roots of $G$. Let $N$ be a positive integer. Then 
\[
  \lim_{N\to\infty} \frac{\# {\hw^\vee}^{-t}(N\lambda)}{\vol_{{\hw^\vee}^{-t}(N \lambda)} \left((\hw^\vee)^{-t}_{\mathbb{R}}(N\lambda)\right)} = 1.
\] 
Also,
\[
  \vol_{{\hw^\vee}^{-t}(N\lambda)} \left((\hw^\vee)^{-t}_{\mathbb{R}}(N\lambda)\right)= N^m\vol_{{\hw^\vee}^{-t}(\lambda)} \left((\hw^\vee)^{-t}_{\mathbb{R}}(\lambda)\right).
\]
Therefore,
\begin{align*}
  \vol_{{\hw^\vee}^{-t}(\lambda)} \left((\hw^\vee)^{-t}_{\mathbb{R}}(\lambda)\right)&= \lim_{N\to\infty}\frac{1}{N^m} \prod_{\alpha>0}\frac{(N\lambda+\rho,\alpha)}{(\rho,\alpha)}\\
  & = \prod_{\alpha>0}\lim_{N\to\infty}\left(\frac{\left(\lambda,\alpha\right)}{(\rho,\alpha)} +\frac{1}{N}\right) \\
  & = \prod_{\alpha>0}\frac{\left(\lambda,\alpha\right)}{(\rho,\alpha)} .
\end{align*}
It is well known that
\begin{equation}
\label{KKSvolume}
 \vol_{\omega_{\mathfrak{k}^*}} (\mathcal{O}_{\mathfrak{k}^*}(\sqrt{-1}\lambda))= (2\pi)^m \prod_{\alpha>0}\frac{\left(\lambda,\alpha\right)}{(\rho,\alpha)},
\end{equation}
see for instance Section 3.5 of \cite{Kirillov}.

Combining \eqref{comparevols}, \eqref{comparevol2}, and \eqref{KKSvolume}, we get the result.
\end{proof}

\begin{Cor}\label{Vol}
    For all $\lambda^\vee\in \psi^{-1}(X^*_+(H)+\rho)$, one has
  \[
  \dim V_{\lambda-\rho}=\left(\prod_{\alpha>0} d_{\alpha}\right)\cdot  \dim V_{\lambda^\vee-\rho^\vee}, \]
  where  $ d_{\alpha}=\frac{2}{(\alpha,\alpha)}$ and $\lambda=\psi(\lambda^\vee)$.
\end{Cor}
\begin{proof}
    Given a reduced word $\mathbf{i}=(i_1,\dots,i_m)$ of the longest element $w_0$, positive roots can be written in the following order:
    \[
        \alpha_{i_1},\ s_{i_1}\alpha_{i_2},\ \dots,\  s_{i_1}\cdots s_{i_{m-1}}\alpha_{i_m}.
    \]
    Since the bilinear form is $W$-invariant, for positive root $\alpha=s_{i_1}\cdots s_{i_{j-1}}\alpha_{i_j}$, we get:
    \[
        (\alpha,\alpha)=(\alpha_{i_j},\alpha_{i_j}).
    \]
    Then one has $\prod_{\alpha>0} d_\alpha= \prod_{j=1}^{m} d_{i_j}$.
  By Theorem \ref{theoremcomparevolume} and its proof, we get
  \begin{align*}
    \dim V_{\lambda-\rho}&=
    \vol_{{\hw^\vee}^{-t}(\lambda)} \left((\hw^\vee)^{-t}_{\mathbb{R}}(\lambda)\right).
    \end{align*}
    Taking the determinant of $\psi_{\sigma(\mathbf{i})}^{\mathbb{R}}$, one finds
    \begin{align*}
        \vol_{{\hw^\vee}^{-t}(\lambda)} \left((\hw^\vee)^{-t}_{\mathbb{R}}(\lambda)\right) & =\left( \prod_{\alpha>0} d_{\alpha}\right)\cdot \vol_{\hw^{-t}(\lambda^\vee)} \left(\hw^{-t}_{\mathbb{R}}(\lambda^\vee)\right) \\
    & =\left(\prod_{\alpha>0} d_{\alpha}\right)\cdot\dim V_{\lambda^\vee-\rho^\vee}. \tag*{\qedhere}
  \end{align*}
\end{proof}

In the following, we present a direct proof of Corollary \ref{Vol}. Let $\psi(\lambda^\vee)=\lambda\in X^*_+(H)$ and denote $\rho=\frac{1}{2}\sum_{\alpha>0} \alpha$ and $\rho^\vee=\frac{1}{2}\sum_{\alpha>0} \alpha^\vee$ as before. Note  $\psi$ preserves the bilinear forms on $\mathfrak{h}$ and $\mathfrak{h}^*$ and commutes with the $W$-action. 

\begin{Lem} \label{le:denom}
	For each complex semisimple $\mathfrak{g}$ one has for a formal parameter $q$,
	\begin{equation*}
		\prod\limits_{\alpha>0} (q^{\frac{1}{2}\langle\rho^\vee,\alpha\rangle}-q^{-\frac{1}{2}\langle\rho^\vee,\alpha\rangle})=\prod\limits_{\alpha^\vee>0} (q^{\frac{1}{2}\langle \alpha^\vee,\rho\rangle}-q^{-\frac{1}{2}\langle \alpha^\vee,\rho\rangle}).
	\end{equation*}
	In particular, $\prod\limits_{\alpha>0} \langle \rho^\vee,\alpha\rangle=\prod\limits_{\alpha^\vee>0}\langle \alpha^\vee,\rho\rangle$.
\end{Lem}
\begin{proof}
	Note we have the following Weyl denominator Formula:
	\[
		e^\rho\prod_{\alpha>0}(1-e^{-\alpha})=\sum_{w\in W} (-1)^{\ell(w)}e^{w\rho}.
	\]
	Let $Q$ be the root lattice and $Q^\vee$ be the coroot lattice of $\mathfrak{g}$. Applying the ring homomorphisms
	\[
		\eta\colon \mathbb{Z}[\frac{1}{2}Q]\to \mathbb{Z}[q^{\pm \frac{1}{2}}]\ :\ e^{\beta}\mapsto q^{\langle\rho^\vee,\beta\rangle}; \quad \eta^\vee\colon \mathbb{Z}[\frac{1}{2}Q^\vee]\to \mathbb{Z}[q^{\pm \frac{1}{2}}]\ :\ e^{\beta^\vee}=q^{\langle\beta^\vee,\rho\rangle}
	\]
	to the Weyl denominator formula for $Q$ and $Q^\vee$, we obtain
	\[
		\prod_{\alpha>0} (q^{\frac{1}{2}\langle\rho^\vee,\alpha\rangle}-q^{-\frac{1}{2}\langle\rho^\vee,\alpha\rangle})=\sum_{w\in W} (-1)^{\ell(w)}e^{\langle \rho^\vee, w\rho\rangle}=\prod\limits_{\alpha^\vee>0} (q^{\frac{1}{2}\langle \alpha^\vee,\rho\rangle}-q^{-\frac{1}{2}\langle \alpha^\vee,\rho\rangle}). 
	\]
	The second assertion follows by dividing both sides with the appropriate power of $q^{\frac{1}{2}}-q^{-\frac{1}{2}}$ and taking the limit as $q\mapsto 1$.
\end{proof}

For $\lambda^\vee\in \mathfrak{h}$ and $\lambda\in \mathfrak{h}^*$, we rewrite the Weyl dimension formula: If $\lambda\in X^*_+(H)+\rho$, then
\[
	\dim V_{\lambda-\rho}=\prod_{\alpha>0}\frac{(\lambda,\alpha)}{(\rho,\alpha)}=\prod\limits_{\alpha^\vee>0} \frac{\langle \alpha^\vee,\lambda\rangle}{\langle \alpha^\vee,\rho\rangle},
\]
where $V_{\lambda-\rho}$ is the irreducible highest weight module with highest weight $\lambda-\rho$. Then:

\begin{proof}[Proof of Corollary \ref{Vol}]
	Indeed, Lemma \ref{le:denom} implies that 
	\[
		\dim V_{\psi\lambda^\vee-\rho}=\prod_{\alpha^\vee>0} \frac{\langle \alpha^\vee,\psi\lambda^\vee\rangle}{\langle \alpha^\vee,\rho\rangle}=\prod_{\alpha^\vee>0} \frac{\langle \psi\alpha^\vee,\lambda^\vee\rangle}{\langle \alpha,\rho^\vee \rangle}=\prod_{\alpha>0} \frac{d_\alpha\langle \alpha,\lambda^\vee\rangle}{\langle \alpha,\rho^\vee\rangle}=\prod_{\alpha>0} d_{\alpha}\cdot\dim V_{\lambda^\vee-\rho^\vee}.
	\]
	The corollary is proved.
\end{proof}

\begin{appendices}
\section{Example: Duality between \texorpdfstring{$B_2$}{B2} and \texorpdfstring{$C_2$}{C2} }\label{B2vsC2}
\label{appendix1}

Note ${\rm SO}_{2n+1}^\vee={\rm Sp}_{2n}$. Let us focus on the case $n=2$. Here we use an alternative description of ${\rm SO}_5$. Denote 
\[
	J_n=
	\begin{bmatrix}
		  &   & 1 \\
		  & \iddots &  \\
		1 &   &   
	\end{bmatrix}. 
\]
The group ${\rm SO}_5$ is isomorphic to 
\[
	G=\{X\in {\rm GL}_5\mid XJ_5X^T=J_5\},
\]
with Lie algebra:
\[
	\mathfrak{g}=\{x\in \mathfrak{gl}(5)\mid x+J_5x^TJ_5=0\}.
\]
Cartan subalgebra: 
\[
	\mathfrak{h}=\{\diag (x_1,x_2,0,-x_2,-x_1)\}.
\]
Borel subalgebra: 
\[
	\mathfrak{b}=\mathfrak{g}\cap \{\text{upper-triangular matrices}\}.
\]
Cartan matrix and a symmetrizer:
\[
	A=
	\begin{bmatrix}
		2 & -1 \\
		-2 & 2
	\end{bmatrix}=
	\begin{bmatrix}
		1 &  \\
		 & 2
	\end{bmatrix}
	\begin{bmatrix}
		2 & -1 \\
		-1 & 1
	\end{bmatrix}, \quad
	D=
	\begin{bmatrix}
		1 &  \\
		 & 2
	\end{bmatrix},
\]
Orthonormal basis in $\mathfrak{h}^*$: 
\[
	\zeta_i: \diag (x_1,x_2,0,-x_2,-x_1)\mapsto x_i.
\]

Simple roots: 
\[
	\alpha_1=\zeta_1-\zeta_2,\quad \alpha_2=\zeta_2.
\] 
Positive roots: 
\[
	\alpha_1,\quad \alpha_2,\quad \alpha_3:=\alpha_1+\alpha_2,\quad \alpha_4:=\alpha_1+2\alpha_2.
\]
Simple coroots: 
\[
	\alpha_1^\vee=\diag(1,-1,0,1,-1); \quad \alpha_2^\vee=\diag(0,2,0,-2,0).
\]
Simple root vectors: 
\[
	F_1=E_{21}-E_{54}; \quad F_2=E_{32}-E_{43}.
\]
Fundamental weights: 
\[
	\omega_1=\alpha_3, \quad \omega_2=\frac{1}{2}\alpha_4.
\]
Fundamental coweights: 
\[
	\omega_1^\vee=\alpha_1^\vee+\frac{1}{2}\alpha_2^\vee, \quad \omega_2^\vee=\alpha_1^\vee+\alpha_2^\vee.
\]
Character lattice of the maximal torus:
\[
	X^*(H)=\mathbb{Z}\{\alpha_1,\alpha_2\}.
\]
Cocharacter lattice of the maximal torus:
\[
	X_*(H)=\mathbb{Z}\{\omega_1^\vee,\omega_2^\vee\}.
\]
Weyl group: 
\[
	W=S_2\ltimes \mathbb{Z}_2, \text{~with generator~} s_1, s_2 \text{~satisfying~} (s_1s_2)^4=1.
\]
The longest element: 
\[
	w_0=(s_1s_2)^2=(s_2s_1)^2.
\]

Now let us compute the BK potential and the BK cone. Note that the lift of $s_i$ to $G$ is given by:
\[
	\overline{s_1}=P_1P_4;\quad \overline{s_2}=P_2P_3P_2; \quad \text{~where~} P_i = E_{i,i+1}-E_{i+1,i}. 
\]
And note $(\overline{s_1}\overline{s_2})^2=P_1P_2P_3P_4P_1P_2P_3P_1P_2P_1$. Let
\[
	\bm{x}=\exp\left(\ln(x_1)\omega_1^\vee+\ln(x_2)\omega_2^\vee\right)
\]
and
\[
 	x_{-1}(t)=
 	\begin{bmatrix}
		t^{-1} & & & & \\
		1 & t & &  &\\
		& & 1 && \\
		&&& t^{-1} &\\
		&&& -1& t
	\end{bmatrix};\quad
	x_{-2}(t)=
 	\begin{bmatrix}
		1 & & & & \\
		  & t^{-2} & &  &\\
		 & t^{-1} & 1 &  & \\
		 &-\frac{1}{2}&-t &t^2 &\\
		 &&&& 1
	\end{bmatrix}.
\]
Then for the longest word $(s_1s_2)^2$, generic elements of the double Bruhat cell $G^{w_0,e}$ can be written: 
\[
	\bm{x}x_{-1}(t_1)x_{-2}(t_2)x_{-1}(t_3)x_{-2}(t_4)\in G^{w_0,e}.
\]
This element is equal to
\[
	\bm{x}
	\begin{bmatrix}
		\dfrac{1}{t_1 t_3} &  &  &  &  \\[0.4cm]
		\dfrac{t_1}{t_2^2}+\dfrac{1}{t_3} & \dfrac{t_1 t_3}{t_2^2 t_4^2} &  &  &  \\[0.4cm]
		\dfrac{1}{t_2} & \dfrac{t_3}{t_2t_4^2}+\dfrac{1}{t_4} & 1 &  &  \\[0.4cm]
		-\dfrac{1}{2 t_1} & -\dfrac{\left(t_3+t_2 t_4\right){}^2}{2 t_1 t_3 t_4^2} & -\dfrac{t_2 \left(t_3+t_2 t_4\right)}{t_1 t_3} & \dfrac{t_2^2 t_4^2}{t_1 t_3} &  \\[0.4cm]
		\dfrac{1}{2} & \dfrac{1}{2} \left(\dfrac{\left(t_3+t_2 t_4\right){}^2}{t_3 t_4^2}+t_1\right) & \dfrac{t_4 t_2^2}{t_3}+t_2+t_1 t_4 & -\dfrac{\left(t_2^2+t_1 t_3\right) t_4^2}{t_3} & t_1 t_3 \\
	\end{bmatrix}.
\]
Thus the potential is
\[
	\left(t_1+\frac{\left(t_3+t_2 t_4\right)^2}{t_3t_4^2}\right)+t_4+
	x_1\cdot \frac{1}{t_1}+x_2\left(\frac{1}{t_4}+\frac{t_2^2+t_1t_3}{t_2t_3}\right),
\]
which gives us the cone cut out by the inequalities:
\begin{align}\label{SO5}
	\begin{split}
		x_1\geqslant t_1 &\geqslant 0;\\
		x_2 \geqslant t_4 &\geqslant 0;\\
		2t_2\geqslant t_3 \geqslant 2t_4 &\geqslant 0;\\
		x_2 &\geqslant t_2-t_1;\\
		x_2 &\geqslant t_3-t_2.
	\end{split}	
\end{align}

Now let us describe ${\rm Sp}_4$ as the dual of ${\rm SO}_5$. Denote
\[
	J'_{2n}=
	\begin{bmatrix}
		  & J_n \\
		-J_n &     
	\end{bmatrix}. 
	\]
The group ${\rm Sp}_4$ is isomorphic to 
\[
	G^\vee =\{X\in {\rm GL}_4\mid XJ'_4X^T=J'_4\}
\]
with Lie algebra:
\[
	\mathfrak{g}^\vee=\{x\in \mathfrak{gl}(4)\mid x-J'x^TJ'=0\}.
\]
Cartan subalgebra: 
\[
	\mathfrak{h}=\{\diag (x_1,x_2,-x_2,-x_1)\}.
\]
The Borel subalgebra:
\[
	\mathfrak{b}^\vee =\mathfrak{g}\cap \{\text{upper-triangular matrices}\}.
\]
Orthonormal basis in $(\mathfrak{h}^\vee)^*$:
\[
	\zeta_i^\vee: \diag (x_1,x_2,-x_2,-x_1)\mapsto x_i.
\]
Simple roots:
\[
	\beta_1=\zeta_1^\vee-\zeta_2^\vee,\quad \beta_2=2\zeta_2^\vee.
\]
Positive roots:
\[
	\beta_1,\quad \beta_2,\quad \beta_3:=2\beta_1+\beta_2,\quad \beta_4:=\beta_1+\beta_2.
\]
Simple coroots of $\mathfrak{g}^\vee$ are given by:
\[
	\beta_1^\vee=\diag(1,-1,1,-1); \quad \beta_2^\vee=\diag(0,1,-1,0);
\]
Simple root vectors: 
\[
	F_1=E_{21}-E_{43}; \quad F_2=E_{32}.
\]
Fundamental weights:
\[
	\kappa_1=\frac{1}{2}\beta_3, \quad \kappa_2=\beta_4.
\]
Fundamental coweights: 
\[
	\kappa_1^\vee=\beta_1^\vee+\beta_2^\vee, \quad \kappa_2^\vee=\frac{1}{2}\beta_1^\vee+\beta_2^\vee.
\]
Character lattice of the maximal torus:
\[
	X^*(H^\vee)=\mathbb{Z}\{\kappa_1,\kappa_2\}.
\]
Cocharacter lattice of the maximal torus:
\[
	X_*(H)=\mathbb{Z}\{\beta_1^\vee,\beta_2^\vee\}.
\]

To calculate the potential for $G^\vee$, we need the lift of $s_i$ to $G^\vee$:
\[
	\overline{s_1}=P_1P_3;\quad \overline{s_2}=P_2; \quad \text{~where~} P_i = E_{i,i+1}-E_{i+1,i}. 
\]
Note $(\overline{s_1}\overline{s_2})^2=P_1P_2P_3P_1P_2P_1$. Let
\[
	\bm{y}^\vee=\exp\left( \ln(y_1)\beta_1^\vee+\ln(y_2)\beta_2^\vee\right)
\]
and
\[
 	x_{-1}^\vee(t)=
 	\begin{bmatrix}
		t^{-1} &  & & \\
		1  & t & &  \\
		 && t^{-1} &\\
		 &&-1& t
	\end{bmatrix};\quad
	x_{-2}^\vee(t)=
 	\begin{bmatrix}
		1 &  & & \\
		  & t^{-1} &  &  \\
		 &1& t &\\
		 &&& 1
	\end{bmatrix}.
\]
Then for the longest word $(s_1s_2)^2$, generic elements of the double Bruhat cell $G^{\vee;e,w_0}$ can be written: 
\[
	\bm{y}^\vee x_{-1}^\vee(t_1)x_{-2}^\vee(t_2)x_{-1}^\vee(t_3)x_{-2}^\vee(t_4)\in G^{\vee;e,w_0}.
\]
This element is equal to
\[
	\bm{y}^\vee
	\begin{bmatrix}
		\dfrac{1}{t_1 t_3} &  &  &  \\[0.4cm]
		\dfrac{t_1}{t_2}+\dfrac{1}{t_3} & \dfrac{t_1 t_3}{t_2 t_4} &  &  \\[0.4cm]
		\dfrac{1}{t_1} & \dfrac{t_2}{t_1 t_3}+\dfrac{t_3}{t_1 t_4} & \dfrac{t_2 t_4}{t_1 t_3} &  \\[0.4cm]
		-1 & -t_1-\dfrac{t_2}{t_3}-\dfrac{t_3}{t_4} & \left(-t_1-\dfrac{t_2}{t_3}\right) t_4 & t_1 t_3 \\
	\end{bmatrix}.
\]
Thus the potential is
\[
	\left(t_1+\frac{t_2}{t_3}+\frac{t_3}{t_4}\right)+t_4+
	\frac{y_1^2}{y_2}\cdot\frac{1}{t_1}+ \frac{y_2^2}{y_1^2}\left(\frac{(t_1t_3+t_2)^2}{t_2t_3^2}+\frac{1}{t_4}\right),
\]
which gives us the cone cut out by the following inequalities:
\begin{align}\label{SP4}
	\begin{split}
		2y_1-y_2\geqslant t_1 &\geqslant 0;\\
		2y_2-2y_1 \geqslant t_4 &\geqslant 0;\\
		t_2\geqslant t_3 \geqslant t_4 &\geqslant 0;\\
		2y_2-2y_1 &\geqslant t_2-2t_1;\\
		2y_2-2y_1 &\geqslant 2t_3-t_2.
	\end{split}
\end{align}

Recall that $\psi: X_*(H)\to X^*(H)$ is given by:
\[
	x_1\omega_1^\vee+x_2\omega_2^\vee\mapsto (x_1+x_2)\alpha_1+(x_1+2x_2)\alpha_2.
\]
Then the map $\psi_{\mathbf{i}}: \mathcal{L} \to \mathcal{L}^\vee$ is given by:
\[
	(x_1,x_2;t_1,t_2,t_3,t_4)\mapsto (x_1+x_2,x_1+2x_2;t_1,2t_2,t_3,2t_4).
\]
Thus it easy to see, after replacing $(y_1,y_2;t_1,t_2,t_3,t_4)$ by $(x_1+x_2,x_1+2x_2;t_1,2t_2,t_3,2t_4)$, that the real cone defined by \eqref{SP4} is the real cone defined by \eqref{SO5}.

\section{Bohr-Sommerfeld lattices and tropical Poisson varieties}\label{appendix2}

In this section we extend the notion of a Bohr-Sommerfeld lattice to the category of tropical Poisson varieties, which we called $\mathbf{PTrop}$ in \cite{ABHL}. We will actually consider the category $\mathbf{DecPTrop}$ of \emph{decorated tropical Poisson varieties}. An object of $\mathbf{DecPTrop}$ is a tuple $(\mathcal{C}\times T, X^t, \pi, \hw, P)$, where 
\begin{itemize}
  \item $T$ is a connected subgroup of $(S^1)^n$;
  \item $X^t = \Hom(S^1,(S^1)^n)\cong \Hom (\mathbb{C}^\times,(\mathbb{C}^\times)^n)=((\mathbb{C}^\times)^n)^t$;
  \item $\mathcal{C}\subset X^t\otimes \mathbb{R}$ is an open rational polyhedral cone in $X^t\otimes \mathbb{R}^n$;
  \item $\pi$ is a constant Poisson bivector on $\mathcal{C}\times T$, and the projections to $\mathcal{C}$ and $T$ (both equipped with the zero Poisson structure) are Poisson maps;
  \item $P\cong \mathbb{Z}^{m-\dim(T)}$ is a free abelian group;
  \item $\hw\colon X^t\to P$ is a $\mathbb{Z}$-linear map, so that fibers of the induced map $\hw_{\mathbb{R}}\circ \pr_1\colon \mathcal{C}\times T\to P\otimes \mathbb{R}$ are the symplectic leaves of $(\mathcal{C}\times T,\pi)$.
\end{itemize}

An example of a decorated tropical Poisson variety is 
\[
  \left(PT(K^*),(G^{w_0,e})^t, \pi_{PT}, \hw^t, X_*(H)\right).
\]
where tropicalization is taken with respect to the chart $\sigma(\mathbf{i})$ given by \eqref{clusterfor1connected}.

An arrow in $\mathbf{DecPTrop}$ is a pair
\[
  (f,g)\colon (\mathcal{C}\times T,X^t, \pi, \hw, P) \to (\mathcal{C}'\times T',{X'}^t, \pi', \hw', P')
\]
where $f\colon X^t\to {X'}^t$ is a piecewise $\mathbb{Z}$-linear map which is homogeneous in the sense that $f(nx)=nf(x)$ for $n\in \mathbb{Z}_{\geqslant 0}$, and $g\colon P\to P'$ is a $\mathbb{Z}$-linear map so that $\hw'\circ f=g\circ \hw$. We require that $f$ induces a map of cones $f\colon \mathcal{C}\to \mathcal{C}'$, and, on each open linearity  chamber $C\subset \mathcal{C}$ of $f$, the naturally induced map $C\times (S^1)^n\to f(C)\times (S^1)^{n'}$ restricts to a Poisson map $f:C\times T\to f(C)\times T'$.

There is an obvious forgetful functor from $\mathbf{DecPTrop}$ to $\mathbf{PTrop}$. 

\begin{Def}
  For a point $\lambda\in P$, consider the fiber $\mathcal{C}_\lambda:=\hw_{\mathbb{R}}^{-1}(\lambda)$. For each $x\in \mathcal{C}_\lambda\cap X^t$, there is a lattice coming from $\pi$ in the tangent space $T_x \mathcal{C}_\lambda$, as in \eqref{hereisaset}. We can realize this as a subset $\Lambda_x$ of $\mathcal{C}_\lambda$.
  A decorated tropical Poisson variety is \emph{quantizable} if, for any $\lambda\in P$ and any $x,y\in \mathcal{C}_\lambda\cap X^t$, one has $ \Lambda_x=\Lambda_y$. If $\mathcal{C}\times T$ is quantizable, define the \emph{Bohr-Sommerfeld lattice} as
  \[
    \Lambda:= \bigcup_{x\in \hw^{-1}(P)} \Lambda_x.
  \]
\end{Def}
Our example $PT(K^*)$ is quantizable, and that this definition of $\Lambda$ agrees with the one in \eqref{bohrlattice1}.

Because $\pi$ is assumed to be constant, to check that $\mathcal{C}\times T$ is quantizable it is enough to check that, for some $\lambda$ in the image of $\hw$ which is sufficiently far from $0$, that
\begin{equation}\label{criterion}
  X^t\cap \mathcal{C} \subset \Lambda_x.
\end{equation}
\begin{Lem}\label{bohrlatticelemma}
  Let $(f,g)$ be an isomorphism of decorated tropical Poisson varieties:
  \[
    (f,g)\colon (\mathcal{C}\times T,X^t, \pi, hw, P) \to (\mathcal{C}'\times T',{X'}^t, \pi', hw', P').
  \]
  Assume $\mathcal{C}\times T$ is quantizable. Then $\mathcal{C}'\times T'$ is quantizable. If $\Lambda'$ denotes the Bohr-Sommerfeld lattice of $\mathcal{C}'\times T'$, then $f(\Lambda)=\Lambda'$.
\end{Lem}
\begin{proof}
  Without loss of generality assume that $P=P'$ and $g=\id$. First, we show that $\mathcal{C}'\times T'$ is quantizable. Let $\lambda\in P$ be in the image of $\hw$ and far from $0$. Pick some point $x\in \mathcal{C}_\lambda \cap X^t$ which is inside an open linearity chamber $C$ of $f$ and far from the boundary of $C$; note that $C$ is a cone because $f$ is homogeneous. Let $C_\lambda= C\cap \mathcal{C}_\lambda$, and let $L=C_\lambda \times T$. Then $f$ induces a symplectomorphism of $L$ onto its image, and we have $f(\Lambda_x\cap C_\lambda)=\Lambda_{f(x)}'\cap f(C_\lambda)$. Since $f$ is an isomorphism, one has $f(X^t\cap C_\lambda)={X'}^t\cap f(C_\lambda)$. And since $\mathcal{C}\times T$ is quantizable, one has $X^t \cap C_\lambda \subset \Lambda_x \cap C_\lambda$. Therefore ${X'}^t \cap f(C_\lambda)\subset \Lambda_{f(x)}'\cap f(C_\lambda)$. Since we chose $x$ to be far from the boundary of $C$, we can extend to all of $\mathcal{C}'_\lambda$, and we have
  \[
    {X'}^t\cap \mathcal{C}'_\lambda \subset \Lambda_{f(x)}'.
  \]
  By the criterion \eqref{criterion}, since $\pi'$ is constant, this tells us that $\mathcal{C}'\times T$ is quantizable. 

  Now, let $\Lambda'$ be the Bohr-Sommerfeld lattice of $\mathcal{C}'\times T'$. We will show that $f(\Lambda)=\Lambda'$. It is enough to check that, for each open linearity cone $C$ of $f$, that $f(\Lambda\cap C)=\Lambda'\cap f(C)$; one can then extend to the boundary of $C$ and $f(C)$ by linearity. For this it is enough to check that $f(\Lambda_x\cap C)=\Lambda'_{f(x)}\cap f(C)$ for all $x\in C\cap X^t$. But this follows from $f$ inducing a Poisson isomorphism from $C\times T$ to its image.
\end{proof}

\end{appendices}

\Addresses

\end{document}